\documentclass[a4paper]{article}
\usepackage{amsmath}
\usepackage{amssymb}
\usepackage{amsthm}
\usepackage[top=30truemm, left=13truemm, right=13truemm]{geometry}

\usepackage{amscd}
\newtheorem{dfn}{Definition}[section]
\newtheorem{thm}[dfn]{Theorem}
\newtheorem{lem}[dfn]{Lemma}
\newtheorem{rem}[dfn]{Remark}
\newtheorem{cor}[dfn]{Corollary}

\newtheorem{prop}[dfn]{Proposition}

\usepackage{amscd}
\usepackage{graphicx}
\usepackage[all]{xy}

\title{The homotopy groups of the automorphism groups of Cuntz-Toeplitz algebras}
\author{Taro Sogabe \\Graduate School of Science \\
Kyoto University \\
Sakyo-ku, Kyoto 606-8502, Japan 
\thanks{staro@math.kyoto-u.ac.jp}}
\begin{document}

\maketitle

\abstract
The Cuntz-Toeplitz algebra $E_{n+1}$ for $n\geq1$ is the universal C*-algebra generated by $n+1$ isometries with mutually orthogonal ranges. In this paper, we determine the homotopy groups of the automorphism group of $E_{n+1}$.

\section{Introduction}
The Cuntz-Toeplitz algebra $E_{n+1}$ for $n\geq1$ is the universal C*-algebra generated by $n+1$ isometries with mutually orthogonal ranges. In this paper, we investigate the automorphism groups of the Cuntz-Toeplitz algebras and determine their homotopy groups.

The homotopy groups of the automorphism groups are necessary to classify the continuous fields of C*-algebras. However, there are only few classes of C*-algebras whose homotopy groups of the automorphism groups are determined. 
To the best knowledge of the author's, the homotopy groups are  known only for Kirchberg algebras \cite{D0, D2, C}, strongly self-absorbing C*-algebras \cite{DP}, and simple AF-algebras \cite{N, T}. %In the above cases, the calculations of the homotopy groups of an automorphism group are decomposed into two steps.
The rough strategy of computation of the homotopy groups in the previous work is as follows. 
First, we show the weak homotopy equivalence between the automorphism group and the endomorphism semi-group. 
Then we compute the homotopy groups of the endomorphism semi-group from the K-theoretic or KK-theoretic data of the C*-algebra.

We illustrate the strategy in the case of Kirchberg algebras where we have a powerful tool, Kirchberg-Phillips' classification theorem. 
Regarding a continuous map $\rho \colon X \to \operatorname{End}(A\otimes \mathbb{K})$ as an element in $\operatorname{Hom}(A\otimes \mathbb{K}, C(X)\otimes A\otimes \mathbb{K})$, we can associate a KK-class $KK(\rho)\in KK(A, C(X)\otimes A)$ to $\rho$. 
Therefore the homotopical data of $\operatorname{End}(A\otimes \mathbb{K})$ are recovered from the KK-theoretic data, %of the automorphism (resp. endomorphism)
 and we can directly compute the general homotopy sets by the map $[X, \operatorname{Aut}A\otimes \mathbb{K}] \to KK(A, C(X)\otimes A)$ (see \cite[Proposition 5.8, Theorem 4.6]{D0}).

In general, there is no such powerful tool and the homotopy groups are computed for only exceptional classes of C*-algebras. 
The strongly self-absorbing C*-algebras are such examples. Dadarlat and Pennig show in \cite[Theorem 2.3]{DP} that the automorphism group $\operatorname{Aut}D$ is contractible for every unital strongly self-absorbing C*-algebra $D$. 
Using a certain fibration, they determine the homotopy groups of $\operatorname{Aut}(D\otimes \mathbb{K})$. 
In this paper, we use similar fibrations in the case mentioned above.

Let $\{T_i\}_{i=1}^{n+1}$ be the canonical generators of $E_{n+1}$ and let $\operatorname{End}_0E_{n+1}$ be the path component of ${\rm id}_{E_{n+1}}$ of the semi-group of unital endomorphisms of $E_{n+1}$. 
We denote by $e$ the minimal projection $1-\sum_{i=1}^{n+1}T_iT_i^*$. 
Then the map $\operatorname{End}_0E_{n+1}\ni \rho \mapsto \sum_{i=1}^{n+1}\rho (T_i)T_i^* \in U_{E_{n+1}}(1-e)$ is a homeomorphism where $U_{E_{n+1}}$ is the unitary group of $E_{n+1}$. 
Our proof of the main result is based on the the fact that the map $U_{E_{n+1}}\to U_{E_{n+1}}(1-e)$ defined by the right multiplication by $1-e$ gives a fibration with a fibre $S^1$.
\begin{thm}\label{MT}
The homotopy groups of $\operatorname{Aut}E_{n+1}$ are as follows $\colon$
\begin{align*}
\pi_1(\operatorname{Aut}E_{n+1})=\mathbb{Z}_n,\quad\pi_{2k+1}(\operatorname{Aut}E_{n+1})=\mathbb{Z},\quad\pi_{2k}(\operatorname{Aut}E_{n+1})=0,\quad k\geq 1.
\end{align*}

\end{thm}
To prove Theorem \ref{MT}, we show that the inclusion map $\operatorname{Aut}E_{n+1} \to \operatorname{End}_0E_{n+1}$ is a weak homotopy equivalence (Theorem \ref{main}).
%\begin{thm}
%The map $\operatorname{Aut}E_{n+1} \to \operatorname{End}_0^1E_{n+1}$ gives the weak homotopy equivalence.
%\end{thm}
%As a corollary, we have the following exact sequence of pointed sets.
%
\begin{cor}
Let $X$ be a compact CW complex.
The following sequence is an exact sequence of pointed sets and the first 4 terms give an exact sequence of groups $\colon$
\begin{align*}
H^1(X)\to K^1(X)\to [X, \operatorname{Aut}E_{n+1}]\to H^2(X)\to [X, \operatorname{BAut}_eE_{n+1}]\to [X, \operatorname{BAut}E_{n+1}]\to H^3(X).
\end{align*}
The group $\operatorname{Aut}_eE_{n+1}$ is the subgroup of all automorphisms that fix the minimal projection $e \in E_{n+1}$.
\end{cor}

The original motivation of this work is to investigate the structure of continuous fields of the Cuntz algebras beyond Dadarlat's work \cite{D1} using the Cuntz-Toeplitz extensions, and we will hopefully come back to this subject in the near future.
We discuss the group structure of the homotopy sets $[X, \operatorname{Aut}E_{n+1}]$ and $[X, \operatorname{Aut}\mathcal{O}_{n+1}]$ in \cite{r2}.

We organize this paper as follows. 
In section 2, we give some preliminaries to compute the homotopy groups. 
We introduce several fibration sequences with the help of \cite[Lemma 2.8, 2.16 and Corollary 2.9]{DP}.
% and Pimsner--Popa--Voiculescu's non-commutative Weyl--von Neumann type theorem \cite[Theorem 2.10]{PPV}. 
As a consequence, the homotopy groups of the connected component of the endomorphism semi-group, denoted by $\operatorname{End}_0E_{n+1}$, are obtained.

In section 3, we show the weak homotopy equivalence of $\operatorname{End}_0E_{n+1}$ and $\operatorname{Aut}E_{n+1}$. The main ingredient of the proof is Pimsner--Popa--Voiculescu's non-commutative Weyl--von Neumann type theorem.
\section{Preliminaries}
\subsection{Notation and the basic facts of the theory of C*-algebras}
Let $A$ be a unital C*-algebra and let $U_A$ be the group of unitary elements in $A$.
We denote by ${U_0}_A$ the path component of $1_A$ of $U_A$.
 For a non-unital C*-algebra $B$, we denote its unitization by $B^{\sim}$.
The K-groups of $A$ are denoted by $K_i(A)$, $i=0, 1$.
 We denote by $[p]_0$ the class of the projection $p$ in $K_0(A)$, and denote by $[u]_1$ the class of the unitary $u$ in $K_1(A)$.
 Let $SA$ be the suspension of $A$, the set of $A$-valued functions on $[0, 1]$ that vanish at $0$ and $1$.
 %The Bott map is the isomorphism $K_0(A) \to K_1(SA)$ which sends the equivalence class of a projection $p$ to the equivalence class of a unitary $(e^{2\pi it}p+1-p)$. 
%  Let $X$ be a compact connected metrizable space. 
% We identify the topological K-groups $K^i(X)$ with the K-groups of $C(X)$, $K_i(C(X))$ where $C(X)$ is the C*-algebra of the continuous functions on $X$.
 For the K-theory, we refer to \cite{Bl, K}.  
 %We denote by $\tilde{K}^i(X)$ the kernel of the evaluation map ${\rm ev}_{x_0} : K^i(X)\to K^i(\{x_0\})=\mathbb{Z}$, and it is identified with $K_i(C_0(X, x_0))$ where $C_0(X, x_0)$ is the C*-algebra of the continuous functions vanishing at the point $x_0$.
 
For a topological space $Y$ and two elements $y_0$ and $y_1$, we denote $y_0\sim_h y_1$ in $Y$ if there is a continuous path from $y_0$ to $y_1$.
Two unitaries $u,v \in U_A$ are homotopy equivalent if $u\sim_h v$ in $U_A$.
 There is a natural map  $U_A/\sim_h \to K_1(A)$ from the set of homotopy classes of unitaries to the $K_1$-group. 
We say $A$ is $K_1$-injective if the map is injective. For the non-unital C*-algebra $B$, it is $K_1$-injective if the natural map $U_{B^{\sim}}/\sim_h\to K_1(B)$ is injective.
 For example, the algebra $A\otimes \mathbb{K}$ is $K_1$-injective by the definition of the $K_1$-group. 
 
We denote by $\mathbb{K}$ the algebra of compact operators of infinite dimensional separable Hilbert space $H$.
 For $A\otimes \mathbb{K}$, we denote by $\mathcal{M}(A\otimes \mathbb{K})$ the multiplier algebra of $A\otimes \mathbb{K}$, and denote by $\mathcal{Q}(A\otimes \mathbb{K})$ its quotient by $A\otimes \mathbb{K}$.
 We denote the quotient map by $\pi$. 
 We remark that $\mathcal{Q}(A\otimes \mathbb{K})$ is $K_1$-injective,
 (see \cite[Section 1.13]{M}).
 We identify $\mathcal{M}(\mathbb{K})$ with $\mathbb{B}(H)$ where $\mathbb{B}(H)$ is the algebra of the bounded operators on $H$. 
For $A=C(X)$, we denote by  $C^b_{s*}(X, \mathbb{B}(H))$  the set of $\mathbb{B}(H)$-valued bounded continuous functions on $X$ with respect to the strong* operator topology (abbreviated to SOT*).
 This is a realization of the multiplier algebra $\mathcal{M}(C(X)\otimes \mathbb{K})$ 
% For every compact operator $T$ and every bounded net $\{U_i\}$ in $\mathbb{B}(H)$ converging to $U$ with respect to SOT*, the net $\{TU_i\}$ converges to $TU$ in the norm topology. 
% Therefore $C(X)\otimes \mathbb{K}$ is an ideal of $C^b_{s*}(X, \mathbb{B}(H))$, and the induced map from $C^b_{s*}(X, \mathbb{B}(H))$ to $\mathcal{M}(C(X)\otimes \mathbb{K})$ is an isomorphism.
 (see \cite[Proposition 2.57]{RW}).
  
We refer to \cite[Theorem 1]{Hig} for the K-theory of the multiplier algebra, and a generalization of Kuiper's theorem.
\begin{thm}\label{kuiper}
Let $A$ be a unital C*-algebra. Then $U_{\mathcal{M}(A\otimes \mathbb{K})}$ is contractible with respect to the norm topology, and we have
$K_i(\mathcal{M}(A\otimes \mathbb{K}))=0, \;i=1, 2.$
\end{thm}

Let $A$, $B$ and $C$ be  C*-algebras. An extension $C$ of $A$ by $B\otimes \mathbb{K}$ is an exact sequence
\begin{align*} 
0\to B\otimes\mathbb{K}\to C\to A\to 0,
\end{align*}
and the Busby invariant of the extension is the induced map $\tau \colon  A\to \mathcal{Q}(B\otimes \mathbb{K})$. We refer to \cite{Bl} for the definition of the Busby invariant.
The extension is called trivial if the above exact sequence splits.
The extension is called essential if $\tau$ is injective, and called unital if $\tau$ is unital. 
We refer to \cite{Bl} for the basic facts of the theory of extensions of C*-algebras. There are two equivalence relations of unital extensions, the strong unitary equivalence and the weak unitary equivalence.
\begin{dfn}
Let $A$ and $B$ be C*-algebras. 
Two Busby invariants $\tau_i : A\to \mathcal{Q}(B\otimes \mathbb{K})$, $i=1, 2$ are said to be strongly unitarily equivalent if there exists a unitary $U \in U_{\mathcal{M}(B\otimes \mathbb{K})}$ satisfying $\tau_1={\rm Ad}\pi(U)\circ\tau_2$. They are said to be weakly unitarily equivalent if there exists a unitary $u \in U_{\mathcal{Q}(B\otimes \mathbb{K})}$ with $\tau_1={\rm Ad}u\circ\tau_2$.
We denote the strong unitary equivalence by $\sim_{s.u.e}$ and denote the weak unitary equivalence by $\sim_{w.u.e}$. 
We denote $\tau_1\sim_s \tau_2$ if there exists two trivial extensions $\rho_1$ and $\rho_2$ satisfying $\tau_1\oplus\rho_1 \sim_{s.u.e}\tau_2\oplus \rho_2$.
We denote by $\operatorname{Ext}(A, B\otimes \mathbb{K})$ the set of the equivalence classes of the Busby invariants with respect to the equivalence relation $\sim_s$. 
\end{dfn}
We note that the weak unitary equivalence, $\sim_{w.u.e}$ induces the equivalence $\sim_s$(see \cite[Proposition 5.6.4]{Bl}).
In this paper, we deal with the extensions of the Cuntz algebras by $C(X)\otimes \mathbb{K}$.
One has a universal coefficient theorem of ${\rm Ext}$-groups.
\begin{thm}[{\cite[Theorem 23.1.1]{Bl}}]\label{uct}
Let $A$ and $B$ be separable C*-algebras, with $A$ in the bootstrap class.
Then there is an unnaturally splitting short exact sequence
\begin{align*}
0\to\bigoplus_{i=0,1}\operatorname{Ext}_{\mathbb{Z}}^1(K_i(A), K_i(B))\to\operatorname{Ext}(A, B\otimes\mathbb{K})\to\bigoplus_{i=0,1}\operatorname{Hom}(K_i(A), K_{i+1}(B))\to 0.
\end{align*}
If $\bigoplus_{i=0,1}\operatorname{Hom}(K_i(A), K_{i+1}(B))=0$, we have an isomorphism
$\operatorname{Ext}(A, B\otimes\mathbb{K})\to \bigoplus_{i=0,1}\operatorname{Ext}_{\mathbb{Z}}^1(K_i(A), K_i(B))$
that sends a class of Busby invariant $[\tau]$ of an extension $0\to B\otimes\mathbb{K}\to C_{\tau}\to A\to 0$ to the class of group extension of the commutative groups $[K_i(B)\to K_i(C_{\tau})\to K_i(A)]\in \operatorname{Ext}_{\mathbb{Z}}^1(K_i(A), K_i(B))$ for $i=0,1$. 
\end{thm}
Let $E_{n+1}$ be the universal C*-algebra generated by $n+1$ isometries with mutually orthogonal ranges and let $\{T_i\}_{i=1}^{n+1}$ be the canonical generator of $E_{n+1}$.
 It is called the Cuntz-Toeplitz algebra. 
 The closed two-sided ideal generated by the minimal projection $e\colon=1-\sum_{i=1}^{n+1}T_iT_i^*$ is isomorphic to the compact operators $\mathbb{K}$, which is known to be the only closed non-trivial two-sided ideal. 
Consider the full Fock space $\mathcal{F}(\mathbb{C}^{n+1})$ and the left creations $\{T_i\}_{i=1}^{n+1}$ (see \cite[Section 1]{Pim}).
Then one has $\mathbb{K}\subset C^*(\{T_i\}_{i=1}^{n+1})=E_{n+1}\subset\mathbb{B}(\mathcal{F}(\mathbb{C}^{n+1}))$.
In this paper, we frequently identify $\mathbb{K}^{\sim}$ with $\mathbb{K}+\mathbb{C}1_{E_{n+1}}\subset\mathbb{B}(\mathcal{F}(\mathbb{C}^{n+1}))$. 
 Let $\pi \colon E_{n+1}\to \mathcal{O}_{n+1}$ be the quotient map by the ideal $\mathbb{K}$, and let $S_i \colon=\pi(T_i)$. 
 The quotient algebra $\mathcal{O}_{n+1}$ is the universal simple C*-algebra generated by $n+1$ isometries with the relation $ \colon S_j^*S_i=\delta _{ij}, \;1=\sum_{i=1}^{n+1}S_iS_i^*$. 
 We denote by $\mathcal{O}_{\infty}$ the universal C*-algebra generated by the countably infinite isometories with mutually orthogonal ranges. 
 The algebras  $\mathcal{O}_{n+1}$ and $\mathcal{O}_{\infty}$ are called the Cuntz algebras, whose K-groups are the following $\colon$
\begin{align*}
K_0(\mathcal{O}_{n+1})=\mathbb{Z}_n, \quad K_1(\mathcal{O}_{n+1})=0,\quad K_0(\mathcal{O}_{\infty})=\mathbb{Z},\quad K_1(\mathcal{O}_{\infty})=0.
\end{align*}
See \cite[Theorem 3.7, 3.8, Corollary 3.11]{C}.
% By the Bott map, the generator of $K_1(S\mathcal{O}_{n+1})$ is the class of the unitary $e^{2\pi it}1_{\mathcal{O}_{n+1}}$.\\

The Cuntz algebras are the Kirchberg algebras, and they tensorially absorb $\mathcal{O}_{\infty}$, $\mathcal{O}_{n+1}\otimes \mathcal{O}_{\infty}\cong \mathcal{O}_{n+1}$. 
The algebra that tensorially absorbs $\mathcal{O}_{\infty}$ has $K_1$-injectivity by the lemma below.
\begin{lem}[{\cite[Lemma 2.1.7]{Phill}}]\label{pil}
Let $A$ be a unital C*-algebra. Then the natural map $U_{A\otimes \mathcal{O}_{\infty}}/\sim_h\to K_1(A\otimes \mathcal{O}_{\infty})$ is bijective.
 In particular, every unital C*-algebra that tensorially absorbs $\mathcal{O}_{\infty}$ is $K_1$-injective. 
\end{lem} 

\begin{dfn}\label{tau0}
We denote by $\tau_0$  the Busby invariant of the extension
$$0\to\mathbb{K}\to E_{n+1}\to\mathcal{O}_{n+1}\to 0.$$
\end{dfn}
The inclusion map $C(X)\otimes \mathbb{K} \hookrightarrow C(X)\otimes E_{n+1}$ induces the Busby invariant $\tau={\rm id}_{C(X)}\otimes \tau_0 \colon C(X)\otimes \mathcal{O}_{n+1} \hookrightarrow \mathcal{Q}(C(X)\otimes \mathbb{K})$ of the unital essential extension
$$0\to C(X)\otimes \mathbb{K}\to C(X)\otimes E_{n+1} \xrightarrow{{\rm id}_{C(X)}\otimes \pi} C(X)\otimes \mathcal{O}_{n+1}\to 0.$$
Since $\mathbb{K}$ and $E_{n+1}$ are KK-equivalent to $\mathbb{C}$ (see \cite[Theorem 4.4]{Pim}) the above exact sequence induces the following 6-term exact sequence $\colon$
\begin{align*}
\xymatrix{
K_0(C(X))\ar[r]^{-n}& K_0(C(X))\ar[r]^{\rho}&K_0(C(X)\otimes\mathcal{O}_{n+1})\ar[d]\\
K_1(C(X)\otimes\mathcal{O}_{n+1})\ar[u]^{{\rm ind}}&K_1(C(X))\ar[l]^{\rho}&K_1(C(X)).\ar[l]^{-n}
}
\end{align*}
%where we identify the mod $n$ K-groups $\tilde{K}^i(X ; \mathbb{Z}_n)$ with the K-groups $K_i(C_0(X, x_0)\otimes \mathcal{O}_{n+1})$. 
%See \cite[Theorem 6.4]{S}.
%We refer to \cite[Section 8]{RS} for mod $n$ K-theory.

For a pointed topological space $(X, x_0)$, we denote by $\Sigma X$ its reduced suspension with the base point $x_0$. 
For pointed topological spaces $(X, x_0), (Y, y_0)$, we denote the set of the continuous maps from $X$ to $Y$ by ${\rm Map}(X, Y)$ and denote the set of base point preserving continuous maps by ${\rm Map}_0(X, Y)$. We denote the homotopy set ${\rm Map}(X, Y)/\sim_h$ by $[X, Y]$ and denote ${\rm Map}_0(X, Y)/\sim_h$ by $[X, Y]_0$. We remark that if $Y$ is an H-space, the natural map $[X, Y]_0\to [X, Y]$ is bijective.
\begin{lem}\label{L1}
Let $(X, x_0)$ be a based compact Hausdorff space. Then the natural map $$U_{{(C_0(X, x_0)\otimes \mathcal{O}_{n+1})}^{\sim}}/\sim_h \to K_1(C_0(X, x_0)\otimes \mathcal{O}_{n+1})$$ is a surjective isomorphism.
\end{lem}
\begin{proof}
Since $K_1(\mathcal{O}_{n+1})=0$, we have
$K_1(C_0(X, x_0)\otimes \mathcal{O}_{n+1})=K_1(C(X)\otimes \mathcal{O}_{n+1}).$
By Lemma \ref{pil}, the natural map $$[X, U_{\mathcal{O}_{n+1}}]=U_{C(X)\otimes \mathcal{O}_{n+1}}/\sim_h \to K_1(C(X)\otimes \mathcal{O}_{n+1})$$ is an isomorphism. 
Since $U_{\mathcal{O}_{n+1}}$ is an H-space, we have
$$U_{(C_0(X, x_0)\otimes \mathcal{O}_{n+1})^{\sim}}/\sim_h =[X, U_{\mathcal{O}_{n+1}}]_0=[X, U_{\mathcal{O}_{n+1}}].$$
Therefore we have the conclusion.
\end{proof}
\begin{lem}[{\cite[Proposition 6.6]{B}}]\label{br}
Let $A$ be a C*-algebra and let $I$ be a two-sided closed ideal of $A$. 
If $A/I$ and $I$ are $K_1$-injective and the natural map $U_{S(A/I)^{\sim}}\to K_1(S(A/I))$ is surjective, then $A$ is $K_1$-injective.
\end{lem}
We refer to \cite{R} for the surjectivity, the properly infinite full projections and the properly infiniteness of the C*-algebras.
\begin{lem}[{\cite[Exercise 8.9]{R}}]\label{properly inf}
Let $A$ be a unital properly infinite C*-algebra. Then the natural map $U_A/\sim_h \to K_1(A)$ is surjective.
\end{lem}
\begin{lem}[{\cite[Exercise 4.9]{R}}]\label{full}
Let $A$ be a unital C*-algebra, and let $p$ and $q$ be properly infinite full projections. Then there exists a partial isometry $v$ with $p=vv^*, \;q=v^*v$, if and only if $[p]_0=[q]_0$ in $K_0(A)$.
\end{lem} 
We show that the algebra $C(X)\otimes E_{n+1}$ is $K_1$-injective.
\begin{prop}\label{k1}
Let $X$ be a compact Hausdorff space. Then, the map
$$U_{C(X)\otimes E_{n+1}}/\sim_h \to K_1(C(X)\otimes E_{n+1})$$
is an isomorphism.
\end{prop}
\begin{proof}
Surjectivity follows from the fact that $C(X)\otimes E_{n+1}$ is properly infinite and Lemma \ref{properly inf}.  We identify $SC_0(X, x_0)\otimes \mathcal{O}_{n+1}$ with $C_0(\Sigma X, x_0)\otimes \mathcal{O}_{n+1}$. Since $C(X)\otimes \mathcal{O}_{n+1}$ is $K_1$-injective by Lemma \ref{pil} and $C(X)\otimes \mathbb{K}$ is $K_1$-injective, it is sufficient to prove the surjectivity of the natural map $U_{(SC(X)\otimes \mathcal{O}_{n+1})^{\sim}}\to K_1(SC(X)\otimes \mathcal{O}_{n+1})$. 
For the space $Y \colon=[0,1]\times X/(\{0\}\times X\sqcup\{1\}\times X)$, we have $SC(X)=C_0(Y, y_0)$. So we have the conclusion by Lemma \ref{br}. 
\end{proof}
Let $\operatorname{End}E_{n+1}$ be the semi-group of unital $*$-endomorphisms of $E_{n+1}$. 
We topologize  $\operatorname{End}E_{n+1}$ by the point-wise norm topology, and let $\operatorname{End}_0E_{n+1}$ be the path component of ${\rm id}_{E_{n+1}}$ in $\operatorname{End}E_{n+1}$. 
We denote by $\operatorname{End}_eE_{n+1}$ (resp. $\operatorname{Aut}_eE_{n+1}$) the subset of $\operatorname{End}E_{n+1}$ (resp. $\operatorname{Aut}E_{n+1}$) consisting of all elements fixing the minimal projection $e$. 
Every automorphism of $E_{n+1}$ preserves the ideal of compact operators  and induces an automorphism of $\mathcal{O}_{n+1}$.
For $\alpha$ in $\operatorname{Aut}E_{n+1}$, we denote by $\tilde{\alpha}$ the induced automorphism of $\mathcal{O}_{n+1}$. This gives a group homomorphism $\operatorname{Aut}E_{n+1}\to \operatorname{Aut}\mathcal{O}_{n+1}$.
% and we denote its kernel by $F$.
\begin{lem} \label{end}
The set $\operatorname{End}_0E_{n+1}$ equals to a subset $\{ \rho \in \operatorname{End}E_{n+1}\mid \rho(e)\; is\; a\; minimal\; projection\; of\; \mathbb{K}\}$ of $\operatorname{End}E_{n+1}$, and the map
$$\operatorname{End}_0E_{n+1}\ni \rho \mapsto u_{\rho}\colon =\sum_{i=1}^{n+1}\rho(T_i)T^*_i\in U_{E_{n+1}}(1-e)\colon=\{u(1-e)\in E_{n+1}\mid u\in U_{E_{n+1}}\}$$
is a homeomorphism.
\end{lem}
\begin{proof}
First, we show $\operatorname{End}_0E_{n+1}\subset\{\rho \in \operatorname{End}E_{n+1}\mid \rho(e) \colon {\rm minimal\;projection}\}$ because the converse is trivial. Consequently, the map $\operatorname{End}_0E_{n+1}\ni\rho \mapsto u_{\rho}\in U_{E_{n+1}}(1-e)$ is well-defined. 
If $\rho(e)$ is a minimal projection, there exists a partial isometry $v$ with $vv^*=\rho(e), \;v^*v=e$. Then the unitary $v+u_{\rho}$ is in the path component of $1_{E_{n+1}}$ by the $K_1$-injectivity of $E_{n+1}$. We take a norm continuous path of unitaries $\{u_t\}_{t\in [0,1]}$ in $U_{E_{n+1}}$ from $v+u_{\rho}$ to $1_{E_{n+1}}$, and we have the continuous path $\rho_t\colon T_i\mapsto u_tT_i$ from $\rho$ to $\rm id_{E_{n+1}}$.

Second, we show the map $\operatorname{End}_0E_{n+1}\ni \rho\mapsto u_{\rho}\in U_{E_{n+1}}(1-e)$ is a homeomorphism.
For every $w\in U_{E_{n+1}}(1-e)$, we have the map $\rho_w : T_i\mapsto wT_i$ by the universality of $E_{n+1}$.
The map $U_{E_{n+1}}(1-e)\ni w\mapsto \rho_w \in \operatorname{End}_0E_{n+1}$ is continuous because $\{T_i\}_{i=1}^{n+1}$ is the generator of $E_{n+1}$.
This gives the inverse of the map $\operatorname{End}_0E_{n+1}\ni\rho \mapsto u_{\rho}\in U_{E_{n+1}}(1-e)$.
\end{proof}
%\begin{rem}
%We remark that $\operatorname{End}E_{n+1}$ is not path connected. We take a projection $p=\sum_{1=1}^{n+1}T_1T_iT_i^*T_1^*$ with $[p]_0=(n+1)[1]_0=[1-e]_0$ in $K_0E_{n+1}=\mathbb{Z}$. Since $p$ and $1-e$ are properly infinite full, there is a partial isometry $v$ with $vv^*=p$, $v^*v=1-e$ by Lemma \ref{full},  and we have an endomorphism $\rho :T_i\mapsto vT_i$. The projection $\rho(e)=1-p\geq T_2T_2^*$ is not the minimal projection.  However we show later that $\operatorname{Aut}E_{n+1}$ is path connected as an application of Pimsner--Popa--Voiculescu's absorbing theorem.
%\end{rem}

\subsection{Section algebras and the theory of extensions of C*-algebras}

We use the following elementary fact.
\begin{lem}\label{bun}
Let $A$ be a unital C*-algebra, and let $X$ be a compact metrizable space. Let $\mathcal{P}_1$ and $\mathcal{P}_2$ be principal $\operatorname{Aut}A$ bundles over $X$. Let $A_1$ and $A_2$ be the section algebras of the associated bundles of $\mathcal{P}_1$ and $\mathcal{P}_2$ with fibre $A$ respectively. 
Then $\mathcal{P}_1$ and $\mathcal{P}_2$ are isomorphic if and only if there exists a $C(X)$-linear isomorphism $\varphi \colon A_1\to A_2$.
\end{lem}

Let $\pi\colon\mathcal{M}(C(X)\otimes \mathbb{K})\to \mathcal{Q}(C(X)\otimes \mathbb{K})$ be the quotient map by the ideal $C(X)\otimes \mathbb{K}$.
We need the following technical theorem of the theory of extensions of C*-algebras.
\begin{thm}[{\cite[Theorem 2.10]{PPV}}]\label{PPV}
Let $X$ be a finite CW complex. 
Let $A$ be a separable simple unital C*-algebra, and let $\mu \colon A\to \mathcal{M}(C(X)\otimes \mathbb{K})$ and $\sigma \colon A\to\mathcal{Q}(C(X)\otimes \mathbb{K})$ be unital $*$-homomorphisms. Then $\sigma\oplus \pi\circ\mu$ and $\sigma$ are strongly unitarily equivalent.
\end{thm}
The theorem above is a special case of \cite[Theorem 2.10]{PPV}. Since $A$ is simple, the assumptions for it are satisfied. 
\begin{lem}\label{ext}
Let $X$ be a finite CW complex.
Let $A$ be a separable simple unital C*-algebra. 
Suppose that $A$ has a unital essential trivial extension $\pi\circ\mu$ where $\mu \colon A\hookrightarrow \mathcal{M}(C(X)\otimes \mathbb{K})$ is a unital embedding.
Then two unital essential extensions $\tau _1$ and $\tau _2$ are weakly unitarily equivalent if and only if $[\tau_1]=[\tau_2]\; in \;\operatorname{Ext}(A, C(X)\otimes \mathbb{K})$.
\end{lem}
\begin{proof}

We show that $[\tau_1]=[\tau_2]$ implies $\tau_1\sim_{w.u.e}\tau_2$ as the other implication is always the case. 
By definition, there exists a trivial extension $\pi\circ\rho_i$ such that $\tau_1\oplus \pi\circ\rho_1\sim_{s.u.e}\tau_2\oplus\pi\circ\rho_2$. 
Adding $\pi\circ\mu$ to the both side, we may assume that $\rho_i(1_A)$ is a properly infinite full projection in $\mathcal{M}(C(X)\otimes \mathbb{K})$. 
Since $K_0(\mathcal{M}(C(X)\otimes \mathbb{K}))=0$ and $\rho_i(1_A)$ is properly infinite full, there exists an isometry $V_i$ with $V_iV_i^*=\rho_i(1_A)$. 
Now we have $\tau_1\oplus \pi\circ({\rm Ad}V_1^*\circ\rho_1)\sim_{w.u.e}\tau_2\oplus \pi\circ({\rm Ad}V_2^*\circ\rho_2)$. 
It follows from Theorem \ref{PPV} that $\tau_i\sim_{s.u.e}\tau_i\oplus \pi\circ({\rm Ad}V_i^*\circ\rho_i)$, and we have the conclusion.
\end{proof}
We have the following theorem of Paschke and Valette.
\begin{thm}[{\cite[Proposition 3]{V}, \cite[Theorem 6]{Pas}}]\label{dual}
Let $A$ and $B$ be unital separable C*-algebras, and assume that $A$ is nuclear. Let $\mu \colon A\to \mathcal{M}(\mathbb{K})$ be a unital embedding with $\mu (A)\cap \mathbb{K}=\{0\}$. For the unital $*$-homomorphism $\tau \colon=\pi(1_{B}\otimes \mu)$, we have an isomorphism 
$$\alpha_{\tau} \colon K_1(\tau(A)'\cap \mathcal{Q}(B\otimes \mathbb{K})) \to \operatorname{Ext}(SA, B\otimes \mathbb{K})$$
which sends the class of a unitary $u \in \tau(A)'\cap \mathcal{Q}(B\otimes \mathbb{K})$ to the class of extension $$\tau_u \colon SA \ni (e^{2\pi it}-1)a \mapsto (u-1)\tau(a) \in \mathcal{Q}(B\otimes \mathbb{K}).$$
\end{thm}
The following theorem holds from  the argument of \cite[Section 1, Theorem 1.5]{P}. 
\begin{thm}[{\cite[Section 1]{P}}]\label{sue}
Let $\tau_1$ and $\tau_2$ be unital extensions of $\mathcal{O}_{n+1}$ by $\mathbb{K}$. Then $\tau_1\sim_{s.u.e} \tau_2$ if and only if $\tau_1\sim_h \tau_2$.
\end{thm}
\begin{prop}\label{VP}
Let $X$ be a compact Hausdorff space with $\operatorname{Tor}(K_0(C(X)),\mathbb{Z}_n)=0$, and let $\sigma \colon \mathcal{O}_{n+1}\to \mathcal{Q}(C(X)\otimes \mathbb{K})$ be an arbitrary unital extension.
Then every element of $K_1(\sigma(\mathcal{O}_{n+1})'\cap\mathcal{Q}(C(X)\otimes \mathbb{K}))$ is a $n$-torsion element, and the set  $U_{(\sigma(\mathcal{O}_{n+1})'\cap\mathcal{Q}(C(X)\otimes \mathbb{K}))}$ is contained in the path component  of $1$ of $U_{\mathcal{Q}(C(X)\otimes \mathbb{K})}$.
\end{prop} 
\begin{proof}
By Theorem \ref{uct}, all elements of $\operatorname{Ext}(S\mathcal{O}_{n+1}, C(X)\otimes \mathbb{K})$ are $n$-torsion elements. 
We define $\tau \colon=\pi \circ (1_{C(X)}\otimes \mu)$ where $\mu \colon \mathcal{O}_{n+1}\to\mathcal{M}(\mathbb{K})$ is a unital embedding. 
Since $\mathcal{O}_{n+1}$ is simple, we have $\mu (\mathcal{O}_{n+1})\cap \mathbb{K}=\{0\}$. By Theorem \ref{dual}, we have $$K_1(\tau(\mathcal{O}_{n+1})'\cap\mathcal{Q}(C(X)\otimes \mathbb{K}))\cong \operatorname{Ext}(S\mathcal{O}_{n+1}, C(X)\otimes \mathbb{K}).$$ So we have $[\sigma ^{\oplus n}]=n[\sigma]=[\tau]=0$, and Lemma \ref{ext} gives a unitary $w\in U_{\mathcal{Q}(C(X)\otimes\mathbb{K})}$ with $\sigma^{\oplus n}={\rm Ad}w\circ \tau$. We have an isomorphism
$${\rm Ad}w \colon (\sigma^{\oplus n}(\mathcal{O}_{n+1})'\cap\mathcal{Q}(C(X)\otimes\mathbb{K}))\cong (\tau (\mathcal{O}_{n+1})'\cap\mathcal{Q}(C(X)\otimes\mathbb{K})).$$
We also have $(\sigma(\mathcal{O}_{n+1})'\cap\mathcal{Q}(C(X)\otimes\mathbb{K}))\otimes \mathbb{M}_n \cong (\sigma^{\oplus n}(\mathcal{O}_{n+1})'\cap\mathcal{Q}(C(X)\otimes\mathbb{K}))$. 
So we have $$K_1(\sigma(\mathcal{O}_{n+1})'\cap\mathcal{Q}(C(X)\otimes \mathbb{K}))\cong K_1(\tau(\mathcal{O}_{n+1})'\cap\mathcal{Q}(C(X)\otimes \mathbb{K})).$$ 
Since $K_0(C(X))=K_1(\mathcal{Q}(C(X)\otimes \mathbb{K}))$ has no $n$-torsion, we have $[w]_1=0\in K_1(\mathcal{Q}(C(X)\otimes \mathbb{K}))$ for every $w\in U_{(\sigma(\mathcal{O}_{n+1})'\cap\mathcal{Q}(C(X)\otimes \mathbb{K}))}$.
So we have the conclusion by $K_1$-injectivity of $\mathcal{Q}(C(X)\otimes \mathbb{K})$ (see \cite[Section 1.13]{M}).
\end{proof} 
As an application of Lemma \ref{ext}, we show in Proposition \ref{pac} that the group $\operatorname{Aut}E_{n+1}$ is path connected.
The straightforward computation yields the lemma below.
\begin{lem}\label{isoke}
We have the following isomorphisms of K-groups and Ext-groups : 
\begin{align*}
%(1_{C(S^{2m-1})}\otimes {\rm id}_{\mathcal{O}_{n+1}})^* &\colon \operatorname{Ext}(C(S^{2m-1})\otimes \mathcal{O}_{n+1}, \mathbb{K})\to \operatorname{Ext}(\mathcal{O}_{n+1}, \mathbb{K}),\\
{{\rm ev}_{pt}}_* &\colon \operatorname{Ext}(\mathcal{O}_{n+1}, C(S^{2m-1})\otimes \mathbb{K})\to \operatorname{Ext}(\mathcal{O}_{n+1}, \mathbb{K}),\\
K_1({{\rm ev}_{pt}}) &\colon K_1(\mathcal{Q}(C(S^{2m-1})\otimes \mathbb{K})) \to K_1(\mathcal{Q}(\mathbb{K})),\\
K_1({{\rm ev}_{pt}}) &\colon K_1(\mathcal{Q}(C([0,1])\otimes \mathbb{K})) \to K_1(\mathcal{Q}(\mathbb{K})),
\end{align*}
for $m\geq 1$.
%where $\mathbb{D}^l$ is a closed $l$-disc.
\end{lem}
We need the following lemma.
\begin{lem}[{\cite[Lemma 2.3]{L}}]
Let $\tau_0 \colon\mathcal{O}_{n+1}\to \mathcal{Q}(\mathbb{K})$ be the Busby invariant in Definition \ref{tau0}.
If a unitary $u$ in $U_{\mathcal{M}(\mathbb{K})}$ commutes with $E_{n+1}$ up to compact operators (i.e. $[u, d]\in \mathbb{K}$ for every $d$ in $E_{n+1}$), there exists a self adjoint element $h$ in $\mathcal{Q}(\mathbb{K})$ such that $e^{2\pi ih}=\pi(u)$ and $[h, a]=0$ for every $a$ in $\mathcal{O}_{n+1}$.
\end{lem}
\begin{cor}\label{N}
The group $N\colon=\{u\in U_{\mathcal{M}(\mathbb{K})}\mid [u, E_{n+1}]\subset \mathbb{K}\}$ is path connected.
\end{cor}
\begin{lem}\label{impiuni}
Let $\alpha$ be an automorphism of $E_{n+1}$, and $U_{\alpha}\colon=\sum_i\alpha(e_{i1})ve_{1i}$ be an implementing unitary of $\alpha\restriction_{\mathbb{K}}$ where $v$ is a partial isometry with $vv^*=\alpha(e), \; v^*v=e$.
Then we have $\alpha={\rm Ad}U_{\alpha}\restriction_{E_{n+1}}$
\end{lem}
\begin{proof}
We show ${\rm Ad}U_{\alpha}\restriction_{E_{n+1}}=\alpha$.
Let $F\subset\mathbb{K}$ be the set of all finite rank projections.
Since $\alpha$ is an automorphism, the image $\alpha(\mathbb{K})=\mathbb{K}$ contains a net $\{\alpha(p)\}_{p\in F}$ that weakly converges to $1$.
For every $d\in E_{n+1}$, we have $\alpha(p)\alpha(d)=\alpha(pd)={\rm Ad}U_{\alpha}(pd)=\alpha(p){\rm Ad}U_{\alpha}(d)$, and $\alpha={\rm Ad}U_{\alpha}\restriction_{E_{n+1}}$ holds.
\end{proof} 
\begin{prop}\label{pac}
The group $\operatorname{Aut}E_{n+1}$ is path connected.
\end{prop}

\begin{proof}
Let $\alpha$ be an automorphism of $E_{n+1}$ and let $\tilde{\alpha}$ be an induced automorphism of $\mathcal{O}_{n+1}$. Since $\operatorname{Aut}\mathcal{O}_{n+1}$ is path connected, we take a path $h_t$ with $h_0=\tilde{\alpha}, h_1={\rm id}_{\mathcal{O}_{n+1}}$.
We take two unital essential extensions
$$\tau_1\colon={\rm id}_{C[0,1]}\otimes \tau_0 \colon C[0,1]\otimes \mathcal{O}_{n+1}\ni f(t) \mapsto f(t)\in\mathcal{Q}(C[0,1]\otimes \mathbb{K})$$
$$\tau_2\colon=\tau_1\circ h \colon C[0,1]\otimes \mathcal{O}_{n+1}\ni f(t) \mapsto h_t(f(t))\in\mathcal{Q}(C[0,1]\otimes \mathbb{K}),$$
where $\tau_0$ is  the Busby invariant in Definition \ref{tau0}, and we regard $h\colon C[0,1]\otimes\mathcal{O}_{n+1}\to C[0,1]\otimes\mathcal{O}_{n+1}$ as a $C[0,1]$-linear isomorphism.
Since $\tau_1\sim_h\tau_2$, we have $[\tau_1]=[\tau_2]$ in $\operatorname{Ext}(C[0,1]\otimes\mathcal{O}_{n+1}, C[0,1]\otimes \mathbb{K})$. 
We have $[\tau_1(1_{C[0,1]}\otimes {\rm id}_{\mathcal{O}_{n+1}})]=[\tau_2(1_{C[0,1]}\otimes {\rm id}_{\mathcal{O}_{n+1}})]$ in $\operatorname{Ext}(\mathcal{O}_{n+1}, C[0,1]\otimes \mathbb{K})$.
By Lemma \ref{ext}, there exists a unitary $v \in U_{\mathcal{Q}(C[0,1]\otimes \mathbb{K})}$ with $\tau_2(1_{C[0,1]}\otimes {\rm id}_{\mathcal{O}_{n+1}})={\rm Ad}v\circ\tau_1(1_{C[0,1]}\otimes {\rm id}_{\mathcal{O}_{n+1}})$.
Since $C[0,1]$ is in the center of $\mathcal{Q}(C[0,1]\otimes\mathbb{K})$, we have $\tau_2={\rm Ad}v\circ\tau_1$.

We show $[v]_1=0$ in $K_1(\mathcal{Q}(C[0,1]\otimes \mathbb{K}))$. By the construction of $\tau_2$ and $v$, the unitary $v_1$ is in $\tau_0(\mathcal{O}_{n+1})'\cap\mathcal{Q}(\mathbb{K})$. By Proposition \ref{VP}, we have  $[v_1]_1=0$ in $K_1(\mathcal{Q}(\mathbb{K}))$. Since the map ${{\rm ev}_1}_*\colon K_1(\mathcal{Q}(C[0,1]\otimes\mathbb{K})) \to K_1(\mathcal{Q}(\mathbb{K}))$ is an isomorphism from Lemma \ref{isoke}, we have $[v]_1=0$.

We take a unitary lift $V\in U_{\mathcal{M}(C[0,1]\otimes \mathbb{K})}$ of $v$. 
It follows that ${\rm Ad}V$ is a $C[0,1]$-linear isomorphism of $C[0,1]\otimes E_{n+1}$. 
Therefore the map $[0,1]\ni t\mapsto {\rm Ad}V_t\in \operatorname{Aut}E_{n+1}$ is continuous. Let $U_{\alpha}\colon=\sum_{i}\alpha(e_{i1})ve_{1i}$ be an implementing unitary of $\alpha$ restricted to $\mathbb{K}$ where $v$ is a partial isometry satisfying  $vv^*=\alpha(e_{11}), v^*v=e_{11}$, and $\{e_{ij}\}$ is a system of matrix units.
By Lemma \ref{impiuni}, we have ${\rm Ad}U_{\alpha}\restriction_{E_{n+1}}=\alpha$.
We have ${\rm Ad}\pi(V_0)=h_0=\tilde{\alpha}={\rm Ad}\pi(U_{\alpha})$, and it follows that $V_0^*U_{\alpha}$ commutes with $E_{n+1}$ up to compact operators.
By Corollary \ref{N}, the automorphism ${\rm Ad}V_0^*U_{\alpha}$ is in the path component of ${\rm id}_{E_{n+1}}$ in $\operatorname{Aut}E_{n+1}$. Similary there is a continuous path from ${\rm id}_{E_{n+1}}$ to ${\rm Ad}V_1$ in $\operatorname{Aut}E_{n+1}$. Therefore we have $$\alpha\sim_h{\rm Ad}V_1\circ{\rm Ad}V_0^*U_{\alpha}\sim_h{\rm Ad}V_1\sim_h{\rm id}_{E_{n+1}}.$$
\end{proof}

\subsection{Implementing unitaries of $\operatorname{Aut}_{C(X)} (C(X)\otimes E_{n+1})$}
Let $\operatorname{Aut}_{C(X)}(C(X)\otimes E_{n+1})$ be the group of $C(X)$-linear automorphisms of $C(X)\otimes E_{n+1}$. We remark that the homotopy set $[X, \operatorname{Aut}E_{n+1}]$ is identified with the set of homotopy equivalence classes of the elements of $\operatorname{Aut}_{C(X)}(C(X)\otimes E_{n+1})$.

Let $G$ be a compact topological group. We denote by $\operatorname{B}G$ its classifying space, and denote by $\operatorname{E}G$ the universal principal $G$-bundle over $\operatorname{B}G$. 
One realization of those spaces is as follows. 
For a contractible space $X$ equipped with a free $G$ action, the quotient map $X\to X/G$ gives the universal bundle.  
We refer to \cite{H} for the basic facts about the classifying spaces. 

Let $H_1$ be the set vectors of norm $1$ in a separable Hilbert space with the norm topology. 
We identify $H_1$ with the set $\{f\in\operatorname{L}^2[0,1]\mid ||f||_{2}=1\}$. 
There is a map $h_t : H_1 \times [0,1] \to H_1$ that sends $(f, t)$ to $(1_{[0, t]}f+1_{[t, 1]})/||1_{[0, t]}f+1_{[t, 1]}||_2$ where $1_{[a,b]}$ is the characteristic function of $[a,b]$. 
This gives the deformation retraction to the set $\{1_{[0, 1]}\}$, and the space $H_1$ is contractible, (see \cite{RW}). 
The group $S^1$ freely acts on $H_1$ by the scalar multiplication. 
Therefore we can adopt $H_1$ as a model of $\operatorname{E}S^1$. 
We identifies $\operatorname{B}S^1$ with the set consisting of all minimal projections, and the map $\operatorname{E}S^1=H_1\ni \xi \mapsto \xi\otimes\xi^*\in\operatorname{B}S^1$ gives the universal bundle where we denote by $\xi \otimes \eta^*$ the operator $H\ni x\mapsto \langle x, \eta\rangle\xi\in H$ for $\xi, \eta \in H$. 
The space $\operatorname{B}S^1$ is the Eilenberg-Maclane space $K(\mathbb{Z}, 2)$ and we identifies the homotopy set $[X, \operatorname{B}S^1]$ with $H^2(X)$ via the Chern classes of the line bundles.
\begin{prop}\label{imp}
Let $X$ be a compact Hausdorff space, and let $\alpha\colon X\to \operatorname{Aut}E_{n+1}$ be a continuous map. Let $\eta$ be the map $\operatorname{Aut}E_{n+1}\ni \alpha\mapsto\alpha(e)\in\operatorname{B}S^1$.
If the image of $[\alpha]$ by the map $[X,\operatorname{Aut}E_{n+1}]\xrightarrow{\eta_*} [X, \operatorname{B}S^1]$ is zero, then there exists a unitary $U$ in $U_{\mathcal{M}(C(X)\otimes \mathbb{K})}$ such that ${\rm Ad}U_x=\alpha_x$.
\end{prop}
\begin{proof}
Let $\xi_0$ be a norm $1$ eigenvector corresponding to the minimal projection $e$.
By assumption, there exists a norm continuous section $\xi \colon X\to H_1$ with $\xi_x\otimes\xi_x^*=\alpha_x(e)$. 
Using a system of matrix units $\{1_{C(X)}\otimes e_{ij}\}$ with $e=e_{11}\colon=\xi_0\otimes\xi_0^*$,
we have a unitary $U_x\colon=\sum_i\alpha_x(e_{i1})\xi_x\otimes\xi_0^*e_{1i}$. 
Since $\xi_x$ is norm continuous, $U\colon X\ni x\mapsto U_x\in\mathcal{M}(\mathbb{K})$ is SOT-continuous.
In particular, $U\colon X\to U_{\mathcal{M}(\mathbb{K})}$ is SOT*-continuous and $U\in U_{\mathcal{M}(C(X)\otimes\mathbb{K})}$.
Lemma \ref{impiuni} shows ${\rm Ad}U_x\restriction_{E_{n+1}}=\alpha_x$ for every $x\in X$.
\end{proof}
\begin{lem}\label{tri}
Let $X$ be a compact Hausdorff space and let $\alpha$ be an element of $\operatorname{Map}(X, \operatorname{Aut}\mathcal{O}_{n+1})$. Then the map $\alpha : C(X)\otimes E_{n+1}\to C(X)\otimes E_{n+1}$ induces the identity map of the K-groups, $K_i(\alpha)={\rm id}_{K_i(C(X)\otimes E_{n+1})}$, $i=1, 2$.
\end{lem}
\begin{proof}
Since $\alpha$ is $C(X)$-linear, we have the commutative diagram below :
$$\xymatrix{
C(X)\ar@{=}[r]\ar[d]&C(X)\ar[d]\\
C(X)\otimes E_{n+1}\ar[r]^{\alpha}&C(X)\otimes E_{n+1}.
}$$
We have the conclusion from the KK-equivalence of $\mathbb{C}$ and $E_{n+1}$.
\end{proof}
Let $r : \operatorname{Aut}E_{n+1}\to \operatorname{Aut}\mathbb{K}$ be the restriction map. Then we have a commutative diagram below 
$$\xymatrix{
[X, \operatorname{Aut}E_{n+1}]\ar[dr]^{\eta_*}\ar[r]^{r_*}&[X, \operatorname{Aut}\mathbb{K}]\ar@{=}[d]^{\eta_*}\\
&[X, \operatorname{B}S^1].
}$$
We remark that the map $\eta : \operatorname{Aut}\mathbb{K}\to \operatorname{B}S^1$ gives the homotopy equivalence (see \cite[Lemma 2.8]{DP}). 
\begin{lem}\label{zero}
The map $\eta_* \colon [S^k, \operatorname{Aut}E_{n+1}]\to [S^k, \operatorname{B}S^1]$ is the zero map for $k\geq1$.
Hence the map $r_*\colon [S^k, \operatorname{Aut}E_{n+1}]\to [S^k, \operatorname{Aut}\mathbb{K}]$ is also zero.
\end{lem}
\begin{proof}
If $k\neq 2$, we have $[S^k, \operatorname{B}S^1]=H^2(S^k)=0$.
We show the statement in the case of $k=2$. For every $\alpha$ in $\operatorname{Map}(S^2,  \operatorname{Aut}E_{n+1})$, the map $K_0(\alpha) :K_0(C(S^2)\otimes E_{n+1}) \to K_0(C(S^2)\otimes E_{n+1})$ is the identity by Lemma \ref{tri}. Since the map $K_0(C(S^2)\otimes \mathbb{K}) \xrightarrow{-n} K_0(C(S^2)\otimes E_{n+1})$ is injective, $[e]_0=[\alpha(e)]_0$ in $K_0(C(S^2)\otimes \mathbb{K})$. The group $\tilde{K}_0C(S^2)$ is generated by the Bott element, the class of the tautological line bundle (see \cite[Section 9.2.10]{Bl}), and the tautological line bundle is also a generator of $H^2(S^2)$. Therefore we have $\eta_*([\alpha])=0$ in $H^2(S^2)$.  
\end{proof}

\subsection{Some fibration sequences}
In this section, we introduce several fibrations to compute the homotopy groups of $\operatorname{Aut}E_{n+1}$.
We refer to \cite[Chap. 6]{AT} for the definition and the basic facts about fibrations.
\begin{dfn}
Let $X, Y$ and $Z$ be the topological spaces, and let $\pi:X\to Y$ be the continuous map.
The map $\pi$ has the homotopy lifting property (abbreviated to HLP) for $Z$, if for every commuting diagram
$$
\xymatrix{
\{0 \} \times Z \ar[r]^{g}\ar[d] &X\ar[d]^{\pi} \\
[0, 1]\times Z\ar[r]^{f}& Y,
}
$$
there exists a continuous map $\tilde{g}\colon [0,1]\times Z\to X$ such that $\tilde{g}(0,z)=g(z)$ for every $z$ in $Z$ and $\pi \circ\tilde{g}=f$.\\
The map $\pi:X\to Y$ is a Serre fibration, if $\pi$ has HLP for every $n$-disc, $\mathbb{D}^n$.
\end{dfn}
We remark that a Serre fibration has HLP for every CW complex. A fibration gives a long exact sequence of  homotopy sets. We denote by $\Omega X$ or $\Omega _{x_0}X$ the loop space of the pointed set $(X, x_0)$.
\begin{thm}
Let $(Z, z_0)$ be a pointed CW complex. 
Let $\pi :(X, x_0)\to (Y, y_0)$ be a Serre fibration with the fibre $F:=\pi^{-1}(y_0)$. Then, there is a long exact sequence of groups $(i\geq 1)$, and exact sequence of pointed sets $(i\geq 0)$
$$\to[Z, \Omega ^iF]_0\to[Z, \Omega^iX]_0\to [Z, \Omega^iY]_0\to \dotsm \to [Z, F]_0\to[Z, X]_0\to [Z, Y]_0.$$
In particular, we have the long exact sequence of the homotopy groups in the case of $Z=\{z_0\}$.
\end{thm}
 We have the following fact.
\begin{prop}[{\cite[Theorem 6.42]{AT}}]\label{sus}
Let $(X, x_0),  (Y, y_0)$ be the pointed topological spaces. Then the natural map $[\Sigma X, Y]_0\to [X, \Omega Y]_0$ is a bijection.
\end{prop}
By the theorem of Hurewicz, every principal $G$-bundle is a fibration. Therefore we use the long exact sequence to compute the homotopy groups of the topological group $G$. We refer to the argument in \cite[Lemma 2.8, 2.16, Corollary 2.9]{DP} for the proof of the following 4 lemmas.%, except Proposition \ref{conn}.  
\begin{lem}\label{LE}
Let $p: U_{E_{n+1}}\to U_{E_{n+1}}(1-e)$ be the multiplication by $1-e$. Then, the map $p$ is a principal $S^1$-bundle that has the $S^1$ action by the right multiplication of $(1-e)+ze$, $z\in S^1$.
\end{lem}
\begin{lem}\label{nemui}
Let $H_1$ be the set of vectors of norm $1$ of the Hilbert space $H$, and $\xi_0\in H_1$ be a vector corresponding the minimal projection $e$. The map $q \colon U_{E_{n+1}}\to H_1$ that sends a unitary $u$ to $u\xi_0$ is a fibration with the fiber $U_{(1-e)E_{n+1}(1-e)}$.
\end{lem}

\begin{rem}\label{hs}
Since $H_1$ is contractible, it is follows from the long exact sequence of the homotopy groups induced by the fibration of Lemma \ref{nemui} that the map $$U_{(1-e)E_{n+1}(1-e)}\ni w\mapsto w+e \in U_{E_{n+1}}$$ is the weak homotopy equivalence. Hence the map $\operatorname{End}_eE_{n+1} \ni \rho \mapsto e+\sum_i{\rho}(T_i)T_i^* \in U_{E_{n+1}}$ is the weak homotopy equivalence.
\end{rem}
\begin{lem}\label{fib}
Let $\eta : \operatorname{End_0}E_{n+1}\to \operatorname{B}S^1$ be the map that sends $\alpha$ to $\alpha(e)$ and let $\operatorname{Aut}_eE_{n+1}$ be the stabilizer subgroup of the minimal projection $e$. Then there is  a principal $\operatorname{Aut}_eE_{n+1}$-bundle  
$$\operatorname{Aut}_eE_{n+1}\to \operatorname{Aut}E_{n+1}\xrightarrow{\eta} \operatorname{B}S^1.$$
\end{lem}
\begin{rem}\label{min}
Since the map $\eta _* : [S^k, \operatorname{Aut}E_{n+1}]\to [S^k, \operatorname{B}S^1]$ is the zero map by Lemma \ref{zero}, it follows from Lemma \ref{fib} that for every $\alpha$ in $\operatorname{Map}(S^k, \operatorname{Aut}E_{n+1})$, there exists $\alpha'$ in $\operatorname{Map}(S^k, \operatorname{Aut}_eE_{n+1})$ that is homotopic to $\alpha$ in $\operatorname{Map}(S^k, \operatorname{Aut}E_{n+1})$.
\end{rem} 

\begin{lem}
The following sequence gives a fibration :
$$\operatorname{End}_eE_{n+1}\to \operatorname{End}_0E_{n+1}\xrightarrow{\eta}\operatorname{B}S^1.$$
\end{lem}
\begin{rem}\label{Ae}
In section 3, we show that the map $\operatorname{Aut}E_{n+1}\to \operatorname{End}_0E_{n+1}$ is a weak homotopy equivalence.
Hence the groups $\operatorname{Aut}_eE_{n+1}$ and $\operatorname{End}_eE_{n+1}$ are weakly homotopy equivalent from the long exact sequences and 5-lemma. 
Then the map $\operatorname{Aut}_eE_{n+1}\ni \alpha \mapsto e+\sum_i\alpha(T_i)T_i^*\in U_{E_{n+1}}$ is the weak homotopy equivalence by the Remark \ref{hs}. 
\end{rem}
By the fibration in Lemma \ref{LE}, we know the homotopy groups of $\operatorname{End}_0E_{n+1}$.
\begin{thm}\label{homotopy}
The homotopy groups of $\operatorname{End}_0E_{n+1}$ are as follows:
$$\pi_1(\operatorname{End}_0E_{n+1})=\mathbb{Z}_n,\;\;\pi_{2k+1}(\operatorname{End}_0E_{n+1})=\mathbb{Z},$$
$$\pi_{2k}(\operatorname{End}_0E_{n+1})=0,\;\;where\; k\geq 1.$$

\end{thm}
\begin{proof}
By Lemma \ref{end}, it is sufficient to compute the homotopy groups of $U_{E_{n+1}}(1-e)$. By the fibration sequence
$$S^1\to U_{E_{n+1}}\xrightarrow{p} U_{E_{n+1}}(1-e),$$
we have the long exact sequence of the homotopy groups 
$$\dotsm\to 0\to \pi_k(U_{E_{n+1}})\to \pi_k(U_{E_{n+1}}(1-e))\to\dotsm \to \pi_1(S^1)\to\pi_1(U_{E_{n+1}})\to\pi_1(U_{E_{n+1}}(1-e))\to 0.$$
The map $S^1\to U_{E_{n+1}}$ sends a complex number $z$ to a unitary $1-e+ze$. 
We have $[S^k, U_{E_{n+1}}]_0=[S^k, U_{E_{n+1}}]=K^1(S^k)$ by $K_1$-injectivity of $C(S^k)\otimes E_{n+1}$. The map$$\mathbb{Z}=[S^1, S^1]_0\ni [z]\mapsto [1-e+ze]\in [S^1, U_{E_{n+1}}]_0=\mathbb{Z}$$ is the multiplication by $-n$, and so we have the conclusion.

\end{proof}
We remark that a generator of $\pi_1(\operatorname{End}_0E_{n+1})=\mathbb{Z}_n$ is the canonical gauge action of $S^1$ that is $\lambda_z : T_i\mapsto zT_i$ for every $z\in S^1$.

\section{The main result}
\subsection{The homotopy groups of $\operatorname{Aut}E_{n+1}$}
In this section, by using the theory of extensions, we show that the inclusion map $\operatorname{Aut}E_{n+1}\to\operatorname{End}_0E_{n+1}$ is a weak homotopy equivalence. 
First, we show $[S^{2m}, \operatorname{Aut}E_{n+1}]$ is trivial, for $m\geq 1$. 
Second, we show the surjectivity of the map $[S^{2m-1}, \operatorname{Aut}E_{n+1}]\to [S^{2m-1}, {\operatorname{End}}_0E_{n+1}]$  for $m\geq 1$. 
Finally, we show the injectivity of the map.

Let $X$ be a compact Hausdorff space.
It is well-known in homotopy theory that every principal $\operatorname{Aut}E_{n+1}$ bundle $\mathcal{P}$ over $X$ comes from the classifying map $X\to \operatorname{BAut}E_{n+1}$ \cite[Section 4, Proposition 10.6]{H}. 
So we identifies the equivalence class of a principal bundle $[\mathcal{P}]$ with the homotopy equivalence class of its classifying map and denote $[\mathcal{P}] \in [X, \operatorname{BAut}E_{n+1}]$. 
For a bundle $\mathcal{P}$, the section algebra of the associated bundle $\mathcal{P}\times_{\operatorname{Aut}E_{n+1}} E_{n+1}$ is a locally trivial continuous $C(X)$-algebra $\Gamma(X, \mathcal{P}\times_{\operatorname{Aut}E_{n+1}} E_{n+1})$.

Let $k$ be a natural number.
 For every $\alpha \in \operatorname{Map}(S^k, \operatorname{Aut}E_{n+1})$, there is a principal $\operatorname{Aut}E_{n+1}$-bundle $\mathcal{P}_{\alpha}$ representing the class $[\alpha]$ in $[S^k, \operatorname{Aut}E_{n+1}]\cong[S^{k+1}, \operatorname{BAut}E_{n+1}]$. 
We construct a continuous field of $E_{n+1}$ over $S^{k+1}$ corresponding to $\mathcal{P}_{\alpha}$ as follows.
We denote the interior of the closed $k+1$-disc by $\overset{\circ}{\mathbb{D}}^{k+1}$. 
We view $S^{k+1}$ as a non-reduced suspension of $S^k$, that is, $\overset{\circ}{\mathbb{D}}^{k+1}\cup S^k\cup \overset{\circ}{\mathbb{D}}^{k+1}$, and view $\alpha$ a clutching function on $S^k$ of two trivial bundles over $\overset{\circ}{\mathbb{D}}^{k+1}\cup S^k$ and $S^k\cup \overset{\circ}{\mathbb{D}}^{k+1}$.  
By the following lemma, we have $[S^k, \operatorname{Aut}E_{n+1}]=[S^{k+1}, \operatorname{BAut}E_{n+1}]$.
\begin{lem}[{\cite[Corollary 8.3]{H}}]\label{BG}
Let $G$ be a path connected group. Let $X$ be a topological space, and let $SX$ be its non-reduced suspension. Then the map
$$[X, G]\ni[\alpha]\mapsto [\mathcal{P}_{\alpha}]\in [SX, \operatorname{B}G]$$
is bijective.
\end{lem}
%We remark that $[X, G]_0=[\Sigma X, \operatorname{B}G]_0$ for the general (not necessary connected ) group $G$ by Proposition \ref{sus}.
\begin{dfn}\label{sect}
We identify the section algebra of $\mathcal{P}_{\alpha}\times _{\operatorname{Aut}E_{n+1}}E_{n+1}$ with the following algebra :
$$B_{\alpha} :=\{(F_1, F_2) \in (C([0, 1]\times S^k)\otimes E_{n+1})^{\oplus 2} \mid F_i(0)\in 1_{C(S^k)}\otimes E_{n+1}, \; F_1(1)=\alpha (F_2(1)) \in C(S^k)\otimes E_{n+1}\},$$
and denote by $C_{\alpha}$ the essential ideal
$$C_{\alpha} :=\{(F_1, F_2) \in (C([0, 1]\times S^k)\otimes \mathbb{K})^{\oplus 2} \mid F_i(0)\in 1_{C(S^k)}\otimes \mathbb{K}, \; F_1(1)=\alpha (F_2(1)) \in C(S^k)\otimes \mathbb{K}\}.$$ 
Let $A_{\alpha}$ be the quotient algebra of $B_{\alpha}$ by $C_{\alpha}$ : 
$$A_{\alpha} :=\{(a_1, a_2)\in (C([0, 1]\times S^k)\otimes \mathcal{O}_{n+1})^{\oplus 2} \mid a_i\in 1_{C(S^k)}\otimes \mathcal{O}_{n+1}, \; a_1(1)=\tilde{\alpha}(a_2(1))\in C(S^k)\otimes \mathcal{O}_{n+1}\}.$$
\end{dfn}
The algebra $A_{\alpha}$ is isomorphic to the section algebra $\Gamma (S^{k+1}, \mathcal{P}_{\tilde{\alpha}}\times_{\operatorname{Aut}\mathcal{O}_{n+1}}\mathcal{O}_{n+1})$, where $\tilde{\alpha}$ is the induced map in $\operatorname{Map}(S^k,\mathcal{O}_{n+1})$. 
We remark that $C(S^{k+1})$ is identified with the algebra
$$\{(f_1, f_2) \in (C([0, 1]\times S^k))^{\oplus 2} \mid f_i(0)\in \mathbb{C}, \; f_1(1)=f_2(1)\in C(S^k)\}.$$
which is the center of $B_{\alpha}$. 
Since the map $[S^k, \operatorname{Aut}E_{n+1}]\to [S^k, \operatorname{Aut}\mathbb{K}]$ is zero map by Lemma \ref{zero}, the associated bundle $\mathcal{P}_{\alpha}\times _{\operatorname{Aut}E_{n+1}}\mathbb{K}$ is trivial. We fix a trivialization and obtain $\theta_{\alpha}\colon C_{\alpha}\to C(S^{k+1})\otimes \mathbb{K}$. Thus we get a unital essential extension $\tau_{\theta_{\alpha}}$
$$
\xymatrix{
C_{\alpha}\ar[r]\ar[d]^{\theta_{\alpha}} &B_{\alpha} \ar[r]^{\pi_{\alpha}}\ar[d]^{\theta_{\alpha}} & A_{\alpha} \ar[d]^{\tau_{\theta_{\alpha}}}\\
C(S^{k+1})\otimes \mathbb{K} \ar[r] &\mathcal{M}(C(S^{k+1})\otimes \mathbb{K}) \ar[r]^{\pi} & \mathcal{Q}(C(S^{k+1})\otimes \mathbb{K})
}
$$ 
where the isomorphism $\theta_{\alpha} : C_{\alpha}\to C(S^{k+1})\otimes \mathbb{K}$ depends on the trivialization of the bundle  $\mathcal{P}_{\alpha}\times _{\operatorname{Aut}E_{n+1}}\mathbb{K}$.
\begin{lem}\label{even}
Let $m\geq 1$ be a natural number. Then we have $[S^{2m}, \operatorname{Aut}E_{n+1}]=0$.
\end{lem}
\begin{proof}
Since $[S^{2m}, \operatorname{Aut}\mathcal{O}_{n+1}]=0$ by \cite[Theorem 7.4]{D2}, there is a trivialization $\varphi_{\alpha} : C(S^{2m+1})\otimes \mathcal{O}_{n+1}\to A_{\alpha}$ that is $C(S^{2m+1})$-linear isomorphism for every $\alpha \in \operatorname{Map}(S^{2m},  \operatorname{Aut}\mathcal{O}_{n+1})$. Consider two extensions of $\mathcal{O}_{n+1}$ by $C(S^{2m+1})\otimes\mathbb{K}$ :
$$\sigma_{\alpha} \colon=\tau_{\theta_{\alpha}}\circ\varphi_{\alpha}(1_{C(S^{2m+1})}\otimes {\rm id}_{\mathcal{O}_{n+1}}),$$
$$\sigma \colon=1_{C(S^{2m+1})}\otimes \tau_0.$$
where the map $\tau_0$ is the Busby invariant in Definition \ref{tau0}.
It follows from the construction that $[{\rm ev}_{pt}\circ\sigma_{\alpha}]=[{\rm ev}_{pt}\circ\sigma]$ in $\operatorname{Ext}(\mathcal{O}_{n+1}, \mathbb{K})$. By Lemma \ref{isoke}, we have $[\sigma_{\alpha}]=[\sigma]$ in $\operatorname{Ext}(\mathcal{O}_{n+1}, C(S^{2m+1})\otimes \mathbb{K})$, and Lemma \ref{ext} yields that there exists a unitary $w$ in $U_{\mathcal{Q}(C(S^{2m+1})\otimes\mathbb{K})}$ satisfying ${\rm Ad}w\circ\sigma=\sigma_{\alpha}$. There is another unitary $U$ in $U_{\mathcal{M}(\mathbb{K})}$ with ${\rm Ad}\pi(U)\circ{\rm ev}_{pt}\circ\sigma={\rm ev}_{pt}\circ\sigma_{\alpha}$ by Theorem \ref{sue}.
 By Proposition \ref{VP}, we have $[{\rm ev}_{pt}(w)]_1=[{\rm ev}_{pt}(w)\pi(U^*)]_1=0$ in $K_1\mathcal{Q}(\mathbb{K})$, and Lemma \ref{isoke} yields that $[w]_1=0$. Therefore we have a unitary $W$ that is a lift of $w$, and the map ${\rm Ad}W : C(S^{2m+1})\otimes E_{n+1} \to \theta_{\alpha}(B_{\alpha})$ is a $C(S^{2m+1})$-linear isomorphism. From Lemma \ref{bun}, the bundle $\mathcal{P}_{\alpha}$ is isomorphic to the trivial bundle, and we have $[\alpha]=0$ in $[S^{2m}, \operatorname{Aut}E_{n+1}]$ by  Lemma \ref{BG} and Proposition \ref{pac}.
\end{proof}
\begin{lem}\label{sur}
Let $m\geq 1$ be a natural number. Then the map 
$$[S^{2m-1}, \operatorname{Aut}E_{n+1}]\ni [\alpha]\mapsto [\tilde{\alpha}]\in [S^{2m-1}, \operatorname{Aut}\mathcal{O}_{n+1}]$$
is surjective. 
\end{lem}
\begin{proof}
Let $\tau$ be the map ${\rm id}_{C(S^{2m-1})}\otimes \tau_0$ where $\tau_0$ is the Busby invariant in Definition \ref{tau0}.
We show that for every $\gamma\in\operatorname{Map}(S^{2m-1}, \operatorname{Aut}\mathcal{O}_{n+1})$ there exists a lift $\Gamma\in\operatorname{Map}(S^{2m-1}, \operatorname{Aut}\mathcal{O}_{n+1})$ with $\tilde{\Gamma}_x=\gamma_x$ for every $x\in S^{2m-1}$. 
We recall the notation that $\tilde{\Gamma}_x$ is an induced automorphism of $\mathcal{O}_{n+1}$ from $\Gamma_x$.
For every $\gamma$ in $\operatorname{Map}(S^{2m-1}, \operatorname{Aut}\mathcal{O}_{n+1})$, we regard $\gamma$ as an element of $\operatorname{Aut}_{C(S^{2m-1})} (C(S^{2m-1})\otimes \mathcal{O}_{n+1})$, and there are two extensions of $\mathcal{O}_{n+1}$
$$\sigma_{\gamma} \colon=\tau\circ\gamma(1_{C(S^{2m-1})}\otimes{\rm id}_{\mathcal{O}_{n+1}}),$$
$$\sigma \colon=\tau(1_{C(S^{2m-1})}\otimes{\rm id}_{\mathcal{O}_{n+1}}).$$
For every $x\in S^{2m-1}$, the map $\gamma_x$ is homotopic to ${\rm id}_{\mathcal{O}_{n+1}}$ in $\operatorname{Aut}\mathcal{O}_{n+1}$ because $\operatorname{Aut}\mathcal{O}_{n+1}$ is path connected by  \cite[Theorem 1.1]{D2}. 
Hence we have ${\rm ev}_x\circ\sigma_{\gamma}\sim_{s.u.e} {\rm ev}_x\circ\sigma$ by  Theorem \ref{sue} because ${\rm ev}_x\circ\sigma_{\gamma}\sim_{h} {\rm ev}_x\circ\sigma$.
From  Lemma \ref{isoke}, we have $[\sigma_{\gamma}]=[\sigma]$ in $\operatorname{Ext}(\mathcal{O}_{n+1}, C(S^{2m-1})\otimes \mathbb{K})$, and $\sigma_{\gamma}\sim_{w.u.e} \sigma$ by  Lemma \ref{ext}. 
We have two unitaries $v\in \mathcal{Q}(C(S^{2m-1})\otimes\mathbb{K})$ and $V\in\mathcal{M}(\mathbb{K})$ satisfying $\sigma_{\gamma}={\rm Ad}v\circ\sigma$ and ${\rm ev}_{pt}\circ\sigma_{\gamma}={\rm Ad}\pi(V)\circ{\rm ev}_{pt}\circ\sigma$.
So we have $[{\rm ev}_{pt}(v)]_1=[\pi(V)^*{\rm ev}_{pt}(v)]_1=0$ by Proposition \ref{VP}, and Lemma \ref{isoke} yields $[v]_1=0\in K_1\mathcal{Q}(C(S^{2m-1})\otimes\mathbb{K})$.
Therefore we have $\sigma_{\gamma}\sim_{s.u.e} \sigma$, and there is a unitary $U_{\gamma}\in U_{\mathcal{M}(C(S^{2m-1})\otimes \mathbb{K})}$ with 
$${\rm Ad}\pi(U_{\gamma})(\tau(1\otimes a)=\tau(\gamma(1\otimes a)), \;a\in \mathcal{O}_{n+1}.$$
We have the following commutative diagram
$$\xymatrix{
C(S^{2m-1})\otimes E_{n+1}\ar[d]^{\pi} \ar[r]^{{\rm Ad}U_{\gamma}}&C(S^{2m-1})\otimes E_{n+1}\ar[d]^{\pi}\\
C(S^{2m-1})\otimes \mathcal{O}_{n+1}\ar[r]^{\gamma}&C(S^{2m-1})\otimes \mathcal{O}_{n+1}.
}$$
The map $\Gamma : S^{2m-1}\ni x\mapsto{\rm Ad}(U_{\gamma})_x \in \operatorname{Aut}E_{n+1}$ is continuous and it is a lift of the map $\gamma$. 
\end{proof}

For $\alpha'$ in $\operatorname{Map}(S^{2m-1}, \operatorname{Aut}_eE_{n+1})$ with $m\geq 1$, we take the map $\theta_{\alpha'}$ as follows. 
Let $U_{\alpha'}$ be a unitary in $U_{\mathcal{M}(C(S^{2m-1})\otimes _{(1-e)}\mathbb{K}_{(1-e)})}$ of the form $U_{\alpha'}=\sum_{i\neq 1}\alpha'(1_{C(S^{2m-1})}\otimes e_{i1})(1_{C(S^{2m-1})}\otimes e_{1i})$ where $\{e_{ij}\}$ is a system of matrix units with $e_{11}=e$. By Theorem \ref{kuiper}, there is a norm continuous path from $1-e$ to $U_{\alpha'}$ in  $U_{\mathcal{M}(C(S^k)\otimes _{(1-e)}\mathbb{K}_{(1-e)})}$. Adding the projection $1_{C(S^{2m-1})}\otimes e$ to the path, we have a norm continuous path $U\in C[0,1]\otimes \mathcal{M}(C(S^{2m-1})\otimes\mathbb{K})$ satisfying
$$U_t \in U_{\mathcal{M}(C(S^{2m-1})\otimes \mathbb{K})}, \; {\rm Ad}U_1\restriction_{E_{n+1}}=\alpha', \; U_0=1, \; eU_t=U_te=e,\; t\in [0, 1],$$
where we write $1_{C(S^{2m-1})}\otimes e$ simply by $e$.
We define two $C(S^{2m})$-algebras $M$ and $M_{\alpha'}$ :
$$M \colon=\{(F_1, F_2) \in \mathcal{M}(C([0,1]\times S^{2m-1})\otimes \mathbb{K})^{\oplus 2}\mid F_i(0) \in 1\otimes \mathcal{M}(\mathbb{K}), \; F_1(1)=F_2(1) \in \mathcal{M}(C(S^{2m-1})\otimes \mathbb{K})\},$$
$$M_{\alpha'} \colon=\{(F_1, F_2) \in  \mathcal{M}(C([0,1]\times S^{2m-1})\otimes \mathbb{K})^{\oplus 2} \mid F_i(0) \in 1\otimes \mathcal{M}(\mathbb{K}), \; F_1(1)={\rm Ad}U_1(F_2(1)) \in \mathcal{M}(C(S^{2m-1})\otimes \mathbb{K}) \}. $$  
The algebras $M$ and $M_{\alpha'}$ are $C(S^{2m})$-linearly isomorphic to $\mathcal{M}(C(S^{2m})\otimes \mathbb{K})$ and we identify $M$ with $\mathcal{M}(C(S^{2m})\otimes \mathbb{K})$.
\begin{dfn}\label{iru}
We define a map $\theta_{\alpha'} : M_{\alpha'}\to M$ by the $C(S^{2m})$-linear isomorphism
$$\theta_{\alpha'}(F_1, F_2):=(F_1, {\rm Ad}U(F_2)), \; F_i\in\mathcal{M}(C(S^{2m-1})\otimes \mathbb{K}).$$
\end{dfn}
The algebras $B_{\alpha'}$ and $C_{\alpha'}$ defined in Definition \ref{sect} are subalgebras of $M_{\alpha'}$.

We denote by $l$ the constant map $S^{2m-1}\to \{{\rm id}_{E_{n+1}}\}$ and denote by $\tilde{l}$ the induced map $S^{2m-1}\to \{{\rm id}_{\mathcal{O}_{n+1}}\}$.   
If $\alpha'$ is homotopic to $l$ in $\operatorname{Map}(S^{2m-1}, \operatorname{End}_eE_{n+1})$, then $[\tilde{\alpha'}]=0$ in $[S^{2m-1}, \operatorname{Aut}\mathcal{O}_{n+1}]$ because of $[S^{2m-1}, \operatorname{End}\mathcal{O}_{n+1}]=[S^{2m-1}, \operatorname{Aut}\mathcal{O}_{n+1}]$ by \cite[Proposition 6.1]{D2}, and there is a trivialization $\varphi_{\alpha'}\colon C(S^{2m})\otimes \mathcal{O}_{n+1}\to A_{\alpha'}$ . 
We can explicitely construct $\varphi_{\alpha'}$ from the homotopy between $\tilde{\alpha'}^{-1}$ and $\tilde{l}$.
\begin{dfn}\label{path}
Let $\alpha'$ be an element of $\operatorname{Map}(S^{2m-1}, \operatorname{Aut}_eE_{n+1})$ homotopic to $l$ in $\operatorname{Map}(S^{2m-1}, \operatorname{End}_eE_{n+1})$, and let $h_t : [0,1]\times S^{2m-1}\to \operatorname{Aut}\mathcal{O}_{n+1}$ be a path from $\tilde{l}=h_0$ to $\tilde{\alpha'}^{-1}=h_1$. Then we define the map $\varphi_{\alpha'}$ as a $C(S^{2m})$-linear isomorphism of the form 
$$\varphi_{\alpha'} \colon C(S^{2m})\otimes \mathcal{O}_{n+1}\ni (a_1(s), a_2(t))\mapsto (a_1(s), h_t(a_2(t)))\in A_{\alpha'},\; s, t \in [0,1],$$
where $C(S^{2m})\otimes \mathcal{O}_{n+1}$ is identified with the algebra
$$A_l=\{(a_1, a_2)\in (C([0,1]\times S^{2m-1})\otimes \mathcal{O}_{n+1})^{\oplus 2}\mid a_i(0)\in 1_{C(S^{2m-1})}\otimes \mathcal{O}_{n+1}, \; a_1(1)=a_2(1) \in C(S^{2m-1})\otimes \mathcal{O}_{n+1}\}.$$
\end{dfn} 
The map $\tau_{\theta_{\alpha'}}\circ\varphi_{\alpha'}$ is the Busby invariant of a unital essential extension of $C(S^{2m})\otimes \mathcal{O}_{n+1}$. The following lemma says that $\tau_{\alpha'}\circ\varphi_{\alpha'}\sim_{w.u.e}\tau={\rm id}_{C(S^{2m})}\otimes \tau_0$ where $\tau_0$ is the Busby invariant in Definition \ref{tau0}. 
\begin{lem}
Let $m\geq 1$ be a natural number and let $\alpha'$ be an element of $\operatorname{Map}(S^{2m-1}, \operatorname{Aut}_eE_{n+1})$ which is homotopic to $l$ in $\operatorname{Map}(S^{2m-1}, \operatorname{End}_eE_{n+1})$. Let $\tau_{\alpha'}\circ\varphi_{\alpha'}$ and $\tau$ be as above. Let $i : C(S^{2m})\ni (f_1, f_2)\mapsto (f_1, f_2) \in B_{\alpha'}$ be the canonical unital embedding and let $j : C(S^{2m})\otimes \mathbb{K} \subset \theta_{\alpha'}(B_{\alpha'})$ be the inclusion map. Then the following hold.

$\rm (1)$ $(\theta_{\alpha'}\circ i)_* : K_0(C(S^{2m}))\cong K_0(\theta_{\alpha'}(B_{\alpha'})).$

We denote $g_1:=(\theta_{\alpha'}\circ i)_*([1_{C(S^{2m})}]_0)$ and $g_2:=(\theta_{\alpha'}\circ i)_*(b_1)$, where $b_1$ is a generator of $K_0C(S^{2m})$.

$\rm (2)$ We have $j_*([1_{C(S^{2m})}\otimes e]_0)=-ng_1$, and there exists a generator $b_2\in K_0(C(S^{2m})\otimes \mathbb{K})$ with $j_*(b_2)=-ng_2$.\\
In particular, it follows that $[\tau_{\alpha'}\circ\varphi_{\alpha'}]=[\tau]$ in $\operatorname{Ext}(C(S^{2m})\otimes \mathcal{O}_{n+1}, C(S^{2m})\otimes \mathbb{K}).$
\end{lem}
\begin{proof}

We identify the sphere $S^{2m}$ with the space $(\overset{\circ}{\mathbb{D}}^{2m}\cup S^{2m-1}\cup \overset{\circ}{\mathbb{D}}^{2m})$ where $\overset{\circ}{\mathbb{D}}^{2m}$ is the interior of the $2m$-disc, and we identify $C_0(\overset{\circ}{\mathbb{D}}^{2m})$ with the algebra $\{ F \in C_0[0, 1)\otimes C(S^{2m-1})\mid F(0)\in \mathbb{C}1_{C(S^{2m-1})}\}$. 
Let $x_0\in S^{2m-1}$ be the base point of $S^{2m}$ and $S^{2m-1}$. 
The map $\iota_0 \colon C_0(\overset{\circ}{\mathbb{D}}^{2m})\ni F\mapsto (F, 0)\in C_0(S^{2m}, x_0)$ induces an isomorphism of K-groups. 
An element $b_1$ is the generator of $K_0(C_0(S^{k+1}, x_0))$.
Let $\iota \colon (C_0(\overset{\circ}{\mathbb{D}}^{2m})\otimes E_{n+1})^{\oplus 2}\ni (F_1, F_2) \mapsto (F_1, F_2) \in B_{\alpha'}$ be an embedding, and let $r \colon B_{\alpha'}\ni (F_1, F_2)\mapsto F_1(1)\in C(S^{2m-1})\otimes E_{n+1}$ be the restriction map.\\ 
First, we show (1). 
We have the following commutative diagram
$$\xymatrix{
(C_0(\overset{\circ}{\mathbb{D}}^{2m})\otimes E_{n+1})^{\oplus 2} \ar[r]^{\;\;\;\;\;\;\;\;\;\;\;\;\;\;\;\iota}&B_{\alpha'}\ar[r]^{r\;\;\;\;\;\;\;}&C(S^{2m-1})\otimes E_{n+1}\\
C_0(\overset{\circ}{\mathbb{D}}^{2m})^{\oplus 2}\ar[r]\ar[u]&C(S^{2m})\ar[r]\ar[u]^{i}&C(S^{2m-1})\ar[u]
}.$$
From  the KK-equivalence of $E_{n+1}$ and $\mathbb{C}$, the vertical maps $({\rm id}_{C_0(\overset{\circ}{\mathbb{D}}^{2m})}\otimes 1_{E_{n+1}})^{\oplus 2}$ and ${\rm id}_{C(S^{2m-1})}\otimes 1_{E_{n+1}}$ induce isomorphisms of K-groups. 
Therefore the map $K_0(i) : K_i(C(S^{2m})) \to K_i(B_{\alpha'})$ is an isomorphism by 6-term exact sequences and the 5-lemma.

Second, we find $b_2$.
We denote by $\iota_1$ the inclusion $C_0(\overset{\circ}{\mathbb{D}}^{2m})\otimes E_{n+1}\ni F_1 \mapsto (F_1, 0)\in (C_0(\overset{\circ}{\mathbb{D}}^{2m})\otimes E_{n+1})^{\oplus 2}$. 
We consider the following commutative diagram
$$\xymatrix{
C_0(\overset{\circ}{\mathbb{D}}^{2m})\otimes \mathbb{K} \ar[r]^{\iota_1}\ar[d]&(C_0(\overset{\circ}{\mathbb{D}}^{2m})\otimes \mathbb{K})^{\oplus 2}\ar[r]^{\;\;\;\;\theta_{\alpha'}\circ\iota}\ar[d]&C_0(S^{2m},x_0)\otimes \mathbb{K}\ar[d]^{j}\\
C_0(\overset{\circ}{\mathbb{D}}^{2m})\otimes E_{n+1}\ar[r]^{\iota_1}&(C_0(\overset{\circ}{\mathbb{D}}^{2m})\otimes E_{n+1})^{\oplus 2}\ar[r]^{\;\;\;\;\;\;\theta_{\alpha'}\circ\iota}&\theta_{\alpha'}(B_{\alpha'})\\
C_0(\overset{\circ}{\mathbb{D}}^{2m})\ar[rr]^{\iota_0}\ar[u]^{{\rm id}\otimes 1_{E_{n+1}}}&&C_0(S^{2m},x_0)\ar[u]^{\theta_{\alpha'}\circ i}
.}$$
Since $\theta_{\alpha'}\circ\iota\circ\iota_1=\iota_0\otimes {\rm id}_{\mathbb{K}}$ and ${K_0(\iota_0)}$ is an isomorphism of K-groups, from diagram chasing we can find a generator $b_2'\in K_0(C_0(\overset{\circ}{\mathbb{D}}^{2m})\otimes\mathbb{K})$ that sent to $-nb_1\in K_0(C_0(S^{2m}, x_0))$ by the map $K_0(\theta_{\alpha'})^{-1}\circ K_0(j)\circ K_0(\theta_{\alpha'}\circ\iota\circ\iota_1)$. Hence we have $b_2\colon=K_0(\theta_{\alpha'}\circ\iota\circ\iota_1)(b_2').$

Third, we show $j_*([1\otimes e]_0)=-ng_1$.
From the assumptions, there exists the map $h' : [0,1]\times S^k \to \operatorname{End}_eE_{n+1}$ with $h'_1=\alpha'$, $h'_0=l$. We have the unital $*$-homomorphism $$\eta : B_{\alpha'}\ni (F_1(s), F_2(t)) \mapsto (F_1(s), h'_t(F_2(t))\in B_l=C(S^{2m})\otimes E_{n+1}$$
which sends $(e, e)\in B_{\alpha'}$ to $(e, e)\in B_l=C(S^{2m})\otimes E_{n+1}$.
We have $\theta_{\alpha'}^{-1}(j(1_{C(S^{2m})}\otimes e))=(e,e)\in B_{\alpha'}$ and $(e, e)=1_{C(S^{2m})}\otimes e\in B_l=C(S^{2m})\otimes E_{n+1}$, and the following commutative diagram holds
$$\xymatrix{
K_0(B_{\alpha'})\ar[r]^{K_0(\eta)}&K_0(B_{l})\\
K_0(C(S^{2m}))\ar[u]^{K_0(i)}\ar@{=}[r]&K_0(C(S^{2m}))\ar[u]^{K_0({\rm id}\otimes 1_{E_{n+1}})}
}.$$
We have  
$$K_0(\eta)([\theta_{\alpha'}^{-1}(j(1_{C(S^{2m})}\otimes e))]_0)=[(e, e)]_0=-n[1_{B_l}]_0=K_0(\eta)([1_{B_{\alpha'}}]_0).$$
Since $i_*$ is an isomorphism, the map $\eta _*$ is also an isomorphism, and we have $[j(1_{C(S^{k+1})}\otimes e)]_0=-n[1_{\theta_{\alpha'}(B_{\alpha'})}]_0=-ng_1$.

Finally, we show that $[{\tau_{\alpha'}}\circ\varphi_{\alpha'}]=[\tau]$.
By Theorem \ref{uct}, we identify $\operatorname{Ext}(C(S^{2m})\otimes \mathcal{O}_{n+1}, C(S^{2m})\otimes \mathbb{K})$ with $\operatorname{Ext}^1_{\mathbb{Z}}(K_0(C(S^{2m})\otimes \mathcal{O}_{n+1}), K_0(C(S^{2m})\otimes \mathbb{K}))$ because $k$ is an odd number. The element $[{\tau_{\alpha'}}\circ\varphi_{\alpha'}]$ is identified with the class of extension 
$$[K_0(C(S^{2m})\otimes \mathbb{K}) \to K_0\theta_{\alpha'}(B_{\alpha'}) \to K_0(C(S^{2m})\otimes \mathcal{O}_{n+1})].$$
By the computation above, it is equal to the class
$[\mathbb{Z}^{\oplus 2}\xrightarrow{-n} \mathbb{Z}^{\oplus 2} \to \mathbb{Z}_n^{\oplus 2}]=[\tau ].$
\end{proof}
We have the Busby invariants of two extensions of $\mathcal{O}_{n+1}$ by $C(S^{2m})\otimes \mathbb{K}$ : 
\begin{align*}
\sigma_{\alpha'} :=&\tau_{\alpha'}\circ\varphi_{\alpha'}(1_{C(S^{2m})}\otimes {\rm id}_{\mathcal{O}_{n+1}}),\\
\sigma :=&1_{C(S^{2m})}\otimes \tau_0.
\end{align*}
From the lemma above, we have 
$$[\sigma_{\alpha'}]=(1\otimes {\rm id}_{\mathcal{O}_{n+1}})^*([\tau_{\alpha'}\circ\varphi_{\alpha'}])=(1\otimes {\rm id}_{\mathcal{O}_{n+1}})^*([\tau])=[\sigma]$$ in $\operatorname{Ext}(\mathcal{O}_{n+1}, C(S^{2m})\otimes \mathbb{K})$.
 Hence there exists a unitary $w_{\alpha'}$ in $U_{\mathcal{Q}(C(S^{2m})\otimes \mathbb{K})}$ satisfying $\sigma_{\alpha'}={\rm Ad}w_{\alpha'}\circ\sigma$ by Lemma \ref{ext}.
 
For a unital essential extension $\nu\colon\mathcal{O}_{n+1}\to\mathcal{Q}(C(S^{2m})\otimes \mathbb{K})$, we take a unitary  $V_{\nu}$ introduced in \cite[Section 1]{P} :
$$
V_{\nu}=\left(
\begin{array}{cccc}
0& &\nu(\mathbb{S}) & \\
\pi(w)& &0_{n+1} &
\end{array}
\right) \in \mathbb{M}_{n+2}(\mathcal{Q}(C(S^{2m})\otimes \mathbb{K})). 
$$
We denote $(S_1, \dotsm, S_{n+1})$ by $\mathbb{S}$. We claim that, for the above $\sigma$, we have ${\rm ind}([V_{\sigma}]_1)=-[1_{C(S^{2m})}\otimes e]_0 \in K_0(C(S^{2m})\otimes \mathbb{K})$. 
Indeed, there is a unitary lift
$$\left(
\begin{array}{cccccccc}
0& &1_{C(S^{2m})}\otimes\mathbb{T}& &1_{C(S^{2m})}\otimes e& &0 &\\
w& &0_{n+1} & &0& &0_{n+1}& \\
& & & &0_1& &w^*& \\
\text{{\huge{O}}}_{n+2} & & & &1_{C(S^{2m})}\otimes\mathbb{T}^*& &0_{n+1}& 
\end{array}
\right)
$$
of $V_{\sigma}\oplus V^*_{\sigma}$, 
where $\mathbb{T}=(T_1,\dotsm, T_{n+1})$ and the element $w$ is of the form $1_{C(S^{2m})}\otimes w_0$ and $w_0 \colon H\to H^{\oplus n+1}$ is a unitary operator.
The element $w\in \mathbb{M}_{n+1,1}(\mathcal{M}(C(S^{2m})\otimes\mathbb{K}))$ is a partial isometry with $ww^*=1_{n+1}$ and $w^*w=1$ in $\mathbb{M}_{n+1}(\mathcal{M}(C(S^{2m})\otimes\mathbb{K}))$.
 From direct computation of the index map, we have ${\rm ind}([V_{\sigma}]_1)=-[1_{C(S^{2m})}\otimes e]_0 \in K_0(C(S^{2m})\otimes \mathbb{K})$. Direct computation yields 
\begin{align*}
V_{\sigma_{\alpha'}}V^*_{\sigma}=&(\sum_{i=1}^{n+1}w_{\alpha'}\sigma(S_i)w^*_{\alpha'}\sigma(S^*_i)) \oplus 1_{n+1}\\
=&\left(
w_{\alpha'}\oplus 1_{n+1}\right)\left(
\begin{array}{cc}
\mathbb{S}&0_1 \\
0_{n+1}&\mathbb{S}^*
\end{array}
\right)\left(
\begin{array}{cc}
w_{\alpha'}^*\otimes 1_{n+1}&0 \\
0&0_1
\end{array}
\right)\left(
\begin{array}{cc}
\mathbb{S}^*&0_{n+1}\\
0_1&\mathbb{S}
\end{array}
\right).
\end{align*}
Hence we have $[V_{\sigma_{\alpha'}}]_1-[V_{\sigma}]_1=-n[w_{\alpha'}]_1$ in $K_1(\mathcal{Q}(C(S^{2m})\otimes \mathbb{K}))$.

We show $[w_{\alpha'}]_1=0 \in K_1(\mathcal{Q}(C(S^{2m})\otimes \mathbb{K}))$ in Theorem \ref{muoo},
and we need the following three lemmas for that.
Recall the path $h \colon [0,1]\times S^{2m-1}\to \operatorname{Aut}\mathcal{O}_{n+1}$ from $\tilde{l}=h_0$ to $\tilde{\alpha'}^{-1}=h_1$ in Definition \ref{path}.
Here and subsequently, we write a unitary $\sum_{i=1}^{n+1}h(1\otimes S_i)1\otimes S_i^*\in U_{(C([0,1]\times S^{2m-1})\otimes \mathcal{O}_{n+1})}$ by $v$ where we denote $1_{C([0,1]\times S^{2m-1})}\otimes S_i$ by $1\otimes S_i$ for simplicity.
We denote 
$$
W=\left(
\begin{array}{cccc}
0&0&& \text{{\huge{O}}}_{n+2}\\
w&0_{n+1}&&\\
&&0_1&w^*\\
\text{{\huge{O}}}_{n+2}&&0&0_{n+1}
\end{array}
\right)\in \mathbb{M}_{2n+4}(M).
$$
By the definition of $\tau_{\theta_{\alpha'}}$ and $\varphi_{\alpha'}$, the following lemma holds.   
\begin{lem}\label{simpl}
Let $y_{\alpha'}$ be an element of the form
$$y_{\alpha'}\colon =\left(\left(
\begin{array}{cccc}
0_1&\mathbb{S}&&\text{{\huge{O}}}_{n+2}\\
0&0_{n+1}&&\\
&&0_1&0\\
\text{{\huge{O}}}_{n+2}&&\mathbb{S}^*&0_{n+1}
\end{array}
\right), \left(
\begin{array}{cccc}
0_1&v\mathbb{S}&&\text{{\huge{O}}}_{n+2}\\
0&0_{n+1}&&\\
&&0_1&0\\
\text{{\huge{O}}}_{n+2}&&\mathbb{S}^*v^*&0_{n+1}
\end{array}
\right)
\right)\in \mathbb{M}_{2n+4}(A_{\alpha'})
$$ 
where we write $1_{C([0,1]\times S^{2m-1})}\otimes \mathbb{S}$ simply by $\mathbb{S}$.
Then we have
$$V_{\sigma_{\alpha'}}\oplus V^*_{\sigma_{\alpha'}}=\pi(W)+\tau_{\theta_{\alpha'}}\otimes {\rm id}_{\mathbb{M}_{2n+4}}(y_{\alpha'}).$$
\end{lem}

In the lemma below, we regard an element $x\in C([0,1]\times S^{2m-1})\otimes E_{n+1}$ as a $C(S^{2m-1})\otimes E_{n+1}$ valued continuous function on $[0,1]$ and denote by $x_t,\; t\in[0,1]$, and frequently write $1_{C(S^{2m-1})\otimes E_{n+1}}$ by $1_{C(S^{2m-1})}$ for simplicity.
\begin{lem}\label{pertur}
Let $V\in U_{C([0,1],\,  C(S^{2m-1})\otimes E_{n+1})}$ be a unitary with $V_0=1_{C(S^{2m-1})\otimes E_{n+1}}$.
Then we can choose a unitary $\mathcal{V}\in U_{C([0,1],\, C(S^{2m-1})\otimes \mathbb{K})^{\sim}}$ satisfying the following
\begin{align*}
\mathcal{V}_0=&1_{C(S^{2m-1})\otimes E_{n+1}},\\
1_{C(S^{2m-1})\otimes E_{n+1}}-&\mathcal{V}_t\in C(S^{2m-1})\otimes \mathbb{K},\\
{\mathcal{V}^*_t V_t(1_{C(S^{2m-1})}\otimes e)}=&{(1_{C(S^{2m-1})}\otimes e)\mathcal{V}^*_t V_t}
=(1_{C(S^{2m-1})}\otimes e),\;t\in [0,1].
\end{align*}
\end{lem}
\begin{proof}
There is a partition $0=t_0<t_1<\dotsm <t_m=1$ satisfying, 
\begin{align*}
||V_t(1_{C(S^{2m-1})}\otimes e)V^*_t-V_{t_k}(1_{C(S^{2m-1})}\otimes e)V^*_{t_k}||< 1,\; t\in [t_k,t_{k+1}].
\end{align*}
We construct the unitary $\mathcal{V}$ by induction.
For $t\in[t_0, t_1]$, we have a polar decomposition
\begin{align}
\label{ir}
V_t(1_{C(S^{2m-1})}\otimes (1-e))V^*_t(1_{C(S^{2m-1})}\otimes (1-e))=w^0_t|V_t(1_{C(S^{2m-1})}\otimes (1-e))V^*_t(1_{C(S^{2m-1})}\otimes (1-e))|
\end{align}
 for $t\in[t_0,t_1],$
and there exists a unitary
\begin{equation*}
\mathcal{V}^0_t\colon =\left\{
\begin{array}{c}
w^0_t+V_t(1_{C(S^{2m-1})}\otimes e),\;t\in[t_0,t_1]\\
w^0_{t_1}+V_{t_1}(1_{C(S^{2m-1})}\otimes e),\;t\in[t_1,t_m].
\end{array}
\right.
\end{equation*}
with $\mathcal{V}^0_0=1_{C(S^{2m-1})}$.
Since $\pi(V_t(1_{C(S^{2m-1})}\otimes (1-e))V^*_t(1_{C(S^{2m-1})}\otimes (1-e)))=1_{C(S^{2m-1})\otimes \mathcal{O}_{n+1}}$ and (\ref{ir}),
we have $1_{C(S^{2m-1})}-\mathcal{V}^0_t\in C(S^{2m-1})\otimes \mathbb{K}.$ 
The unitary ${\mathcal{V}^0}^*V\in U_{C([0,1],\,C(S^{2m-1})\otimes E_{n+1})}$ satisfies the following 
\begin{align}
{\mathcal{V}_t^0}^*V_t(1_{C(S^{2m-1})}\otimes e)=(1_{C(S^{2m-1})}\otimes e){\mathcal{V}_t^0}^*V_t=(1_{C(S^{2m-1})}\otimes e),\;t\in[t_0,t_1],\nonumber\\ 
{\mathcal{V}_0^0}^*V_0=1_{C(S^{2m-1})\otimes E_{n+1}},\nonumber\\
\label{1}
||{\mathcal{V}_t^0}^*V_t(1_{C(S^{2m-1})}\otimes e)V^*_t\mathcal{V}^0_t-{\mathcal{V}^0_{t_k}}^*V_{t_k}(1_{C(S^{2m-1})}\otimes e)V^*_{t_k}\mathcal{V}^0_{t_k}||< 1,\; t\in [t_k,t_{k+1}],\; m-1\geq k\geq 0.
\end{align}
 The condition (\ref{1}) is satisfied by the computation below
\begin{align*}
&{\mathcal{V}^0_t}^*V_t(1_{C(S^{2m-1})}\otimes e)V_t^*\mathcal{V}^0_t-{\mathcal{V}^0_{t_k}}^*V_{t_k}(1_{C(S^{2m-1})}\otimes e)V_{t_k}^*{\mathcal{V}^0_{t_k}}\\
=&\left\{
\begin{array}{c}
0,\;t\in[t_0,t_1]\cap[t_k,t_{k+1}]\\
{\rm Ad}{\mathcal{V}^0_{t_1}}^*(V_t(1_{C(S^{2m-1})}\otimes e)V_t^*-V_{t_k}(1_{C(S^{2m-1})}\otimes e)V_{t_k}^*),\; t\in[t_1,t_m]\cap[t_k,t_{k+1}].
\end{array}\right.
\end{align*}
Let $l$ be a number with $m-1\geq l\geq 0$.
Assume that there exist unitaries $\mathcal{V}^0, \dotsm, \mathcal{V}^{l}$ satisfying 
\begin{align}
1_{C(S^{2m-1})\otimes E_{n+1}}-\mathcal{V}^i_t\in C(S^{2m-1})\otimes \mathbb{K}, \;\mathcal{V}^i_0=1_{C(S^{2m-1})\otimes E_{n+1}},\;l\geq i\geq 0,\nonumber\\
{U^l_t}^*(1_{C(S^{2m-1})}\otimes e)=(1_{C(S^{2m-1})}\otimes e){U^l_t}^*=(1_{C(S^{2m-1})}\otimes e),\;t\in[t_0,t_{l+1}],\nonumber\\
\label{2}
||{U^l_t}^*(1_{C(S^{2m-1})}\otimes e){U^l_t}-{U^l_{t_k}}^*(1_{C(S^{2m-1})}\otimes e){U^l_{t_k}}||< 1,\;t\in [t_k,t_{k+1}], \;m-1\geq k\geq 0,
\end{align}
where we denote $U^l_t\colon =V_t^*\mathcal{V}^0_t\dotsm \mathcal{V}_t^l$. 
Now we construct a unitary $\mathcal{V}^{l+1}$ satisfying 
\begin{align}
\label{4}
1_{C(S^{2m-1})\otimes E_{n+1}}-\mathcal{V}^{l+1}_t\in C(S^{2m-1})\otimes \mathbb{K}, \;\mathcal{V}^{l+1}_0=1_{C(S^{2m-1})\otimes E_{n+1}},\\
\label{5}
{\mathcal{V}^{l+1}_t}^*{U^l_t}^*(1\otimes e)=(1\otimes e){\mathcal{V}^{l+1}_t}^*{U^l_t}^*=(1\otimes e),\;t\in[t_0,t_{l+2}],\\
\label{6}
||{\mathcal{V}_t^{l+1}}^*{U^l_t}^*(1\otimes e){U^l_t}\mathcal{V}^{l+1}_t-{\mathcal{V}^{l+1}_{t_k}}^*{U^l_{t_k}}^*(1\otimes e){U^l_{t_k}}\mathcal{V}^{l+1}_{t_k}||< 1,\;t\in [t_k,t_{k+1}], \;m-1\geq k\geq 0,
\end{align}
where we write $1_{C(S^{2m-1})}\otimes e$ simply by $1\otimes e$.
By the assumption (\ref{2}), we have a partial isometry $w^{l+1}$ of a polar decomposition 
\begin{align}
&(1_{C(S^{2m-1})}-{U^l_t}^*(1\otimes e)U^l_t)(1_{C(S^{2m-1})}-{U^l_{t_{l+1}}}^*(1\otimes e)U^l_{t_{l+1}})\nonumber\\
\label{3}
=&w^{l+1}_t|(1_{C(S^{2m-1})}-{U^l_t}^*(1\otimes e)U^l_t)(1_{C(S^{2m-1})}-{U^l_{t_{l+1}}}^*(1\otimes e)U^l_{t_{l+1}})|,\;t\in [t_{k+1},t_{k+2}]. 
\end{align}
Let $\mathcal{V}^{l+1}$ be a unitary of the form
\begin{align*}
\mathcal{V}^{l+1}_t\colon =\left\{
\begin{array}{c}
1_{C(S^{2m-1})},\;t\in[0,t_{l+1}]\\
w^{l+1}_t+{U^l_t}^*(1_{C(S^{2m-1})}\otimes e)U^l_{t_{l+1}},\;t\in[t_{l+1},t_{l+2}]\\
w^{l+1}_{t_{l+2}}+{U^l_{t_{l+2}}}^*(1_{C(S^{2m-1})}\otimes e)U^l_{t_{l+1}},\;t\in[t_{l+2},t_{m}].
\end{array}
\right.
\end{align*}
Since $\pi((1_{C(S^{2m-1})}-{U^l_t}^*(1\otimes e)U^l_t)(1_{C(S^{2m-1})}-{U^l_{t_{l+1}}}^*(1\otimes e)U^l_{t_{l+1}}))=1_{C(S^{2m-1})\otimes \mathcal{O}_{n+1}}$ and (\ref{3}), we have (\ref{4}).
By the construction of $\mathcal{V}^{l+1}$, we have
\begin{align*}
{\mathcal{V}^{l+1}}^*_t{U^l_t}^*(1\otimes e)=&(1\otimes e){\mathcal{V}^{l+1}}^*_t{U^l_t}^*\\=&(1\otimes e),\;\;\;\;\;t\in[t_0,t_{l+1}],\\
{\mathcal{V}^{l+1}}^*_t{U^l_t}^*(1\otimes e)=&{U^l_{t_{l+1}}}^*(1\otimes e)U^l_t{U_t^l}^*(1\otimes e)\\
=&(1\otimes e)\\
=&(1\otimes e){U^l_{t_{l+1}}}^*(1\otimes e)U^l_{t}{U^l_{t}}^*\\
=&(1\otimes e){\mathcal{V}_t^{l+1}}^*{U^l_t}^*,\;\;\;\;\;t\in[t_{l+1},t_{l+2}],
\end{align*}
and $\mathcal{V}^{l+1}$ satisfies (\ref{5}).
For every $k$, $m-1\geq k\geq 0$, direct computation yields
\begin{align*}
&{\mathcal{V}^{l+1}}^*_t{U^l_t}^*(1\otimes e)U^l_t\mathcal{V}^{l+1}_t-{\mathcal{V}^{l+1}}^*_{t_{k+1}}{U^l_{t_{k}}}^*(1\otimes e)U^l_{t_{k}}\mathcal{V}^{l+1}_{t_{k+1}}\\
=&\left\{
\begin{array}{c}
0,\;\;\;\;\;\;\;t\in[t_0,t_{l+2}]\cap[t_k,t_{k+1}]\\
{\rm Ad}\mathcal{V}^{l+1}_{t_{l+2}}({U^l_t}^*(1\otimes e)U^l_t-{U^l_{t_{k}}}^*(1\otimes e)U^l_{t_k}),\;t\in[t_{l+2},t_m]\cap[t_k,t_{k+1}],
\end{array}\right.
\end{align*}
and the condition (\ref{6}) is satisfied by (\ref{2}).
Now we have a sequence of unitaries $\mathcal{V}^0,\dotsm, \mathcal{V}^{m-1}$ by induction,
and a unitary $\mathcal{V}\colon =\mathcal{V}^0\dotsm\mathcal{V}^{m-1}$ satisfies the assertion of the lemma.
\end{proof}
In the sequal,
we denote by $u_{{\alpha'}^{-1}}$ the element $\sum_{i=1}^{n+1}{\alpha'}^{-1}(T_i)T_i^*$.
The element $u_{{\alpha'}^{-1}}$ is in $U_{C(S^{2m-1})\otimes _{(1-e)}{E_{n+1}}_{(1-e)}}$ because $\alpha'\in \operatorname{Map}(S^{2m-1}, \operatorname{Aut}_eE_{n+1})$. 
We remark that $$\pi(u_{{\alpha'}^{-1}})=\sum_ih_1(1_{C(S^{2m-1})}\otimes S_i)1_{C(S^{2m-1})}\otimes S^*_i=v_1\in C(S^{2m-1})\otimes \mathcal{O}_{n+1}.$$
We also note $v\in {U_0}_{(C([0,1]\times S^{2m-1})\otimes \mathcal{O}_{n+1})}$ because $v_0=1_{C(S^{2m-1})}$.
\begin{lem}\label{implift}
There exists a unitary $V\in U_{C([0,1]\times S^{2m-1})\otimes E_{n+1}}$ satisfying the following
\begin{align*}
\pi(V)=&v,\;V_0=1, \;V_1=u_{{\alpha'}^{-1}}+1_{C(S^{2m-1})}\otimes e,\\
V_t(1_{C(S^{2m-1})}\otimes e)=&(1_{C(S^{2m-1})}\otimes e)V_t=(1_{C(S^{2m-1})}\otimes e),\;t\in[0,1].
\end{align*}
In particular, an element $Y_{\alpha'}$ of the form
$$
Y_{\alpha'}\colon =\left(\left(
\begin{array}{cccc}
0_1&\mathbb{T}&e&0\\
0&0_{n+1}&0&0_{n+1}\\
&&0_1&0\\
\text{{\huge{O}}}_{n+2}&&\mathbb{T}^*&0_{n+1}
\end{array}
\right), \left(
\begin{array}{cccc}
0_1&V\mathbb{T}&e&0\\
0&0_{n+1}&0&0_{n+1}\\
&&0_1&0\\
\text{{\huge{O}}}_{n+2}&&\mathbb{T}^*V^*&0_{n+1}
\end{array}
\right)
\right)\in \mathbb{M}_{2n+4}(B_{\alpha'})
$$
is send to $y_{\alpha'}$ by the quotient map $\pi_{\alpha'}\colon B_{\alpha'}\to A_{\alpha'}$ where we write $1_{C([0,1]\times S^{2m-1})}\otimes \mathbb{T}$ and $1_{C([0,1]\times S^{2m-1})}\otimes e$ by $\mathbb{T}$ and $e$ respectively for simplicity.
\end{lem}
\begin{proof}
Since $v\in {U_0}_{(C([0,1]\times S^{2m-1})\otimes \mathcal{O}_{n+1})}$, one has a unitary lift $V'\in {U_0}_{(C(S^{2m-1})\otimes E_{n+1})}$ of $v$ with $V'_0=1$.
By Lemma \ref{pertur}, we may assume the following
\begin{align}
\label{8}
V'_t(1_{C(S^{2m-1})}\otimes e)=(1_{C(S^{2m-1})}\otimes e)V'_t=1_{C(S^{2m-1})}\otimes e,\;t\in [0,1]\\
\label{7}
V'_0=1_{C(S^{2m-1})\otimes E_{n+1}}
\end{align}
Now we show that we can get the unitary $V$ by a compact perturbation of $V'$.
By (\ref{8}), an element $u_{{\alpha'}^{-1}}{V'_1}^*$ is a unitary in $U_{(C(S^{2m-1})\otimes _{(1-e)}\mathbb{K}_{(1-e)})^{\sim}}$ with $\pi(u_{{\alpha'}^{-1}}{V'_1}^*)=1_{C(S^{2m-1})\otimes \mathcal{O}_{n+1}}$.
Since $\alpha'$ is homotopic to $l$ in $\operatorname{Map}(S^{2m-1},\operatorname{End}_eE_{n+1})$, the unitary $u_{{\alpha'}^{-1}}$ is in ${U_0}_{(C(S^{2m-1})\otimes _{(1-e)}{E_{n+1}}_{(1-e)})}$.
Hence we have $u_{{\alpha'}^{-1}}+1_{C(S^{2m-1})}\otimes e\in{U_0}_{(C(S^{2m-1})\otimes E_{n+1})}$.
Recall that the map $K_1(C(S^{2m-1})\otimes \mathbb{K})\hookrightarrow K_1(C(S^{2m-1})\otimes E_{n+1})$ is injective because $\operatorname{Tor}(K_1(C(S^{2m-1})), \mathbb{Z}_n)=0$ and the map  
\begin{align*}
K_1(C(S^{2m-1})\otimes _{(1-e)}\mathbb{K}_{(1-e)})\ni [u]_1\mapsto [u+1\otimes e]_1\in K_1(C(S^{2m-1})\otimes \mathbb{K})
\end{align*}
is an isomorphism.
Therefore we have $[u_{{\alpha'}^{-1}}{V'_1}^*]_1=0$ in $K_1(C(S^{2m-1})\otimes _{(1-e)}\mathbb{K}_{(1-e)})$,
and we get a continuous path $c\colon [0,1]\to U_{(C(S^{2m-1})\otimes _{(1-e)}\mathbb{K}_{(1-e)})^{\sim}}$ from $u_{{\alpha'}^{-1}}{V'_1}^*$ to $1_{C(S^{2m-1})}\otimes (1-e)$ by the $K_1$-injectivity of $C(S^{2m-1})\otimes_{(1-e)}\mathbb{K}_{(1-e)}$.
For every $t\in[0,1]$, we have $\lambda(t)\in S^1$ with
$$\lambda(0)=\lambda(1)=1,\;\lambda(t)1_{C(S^{2m-1})}\otimes (1-e)-c(t)\in C(S^{2m-1})\otimes _{(1-e)}\mathbb{K}_{(1-e)}$$
by (\ref{7}) and $\pi(u_{{\alpha'}^{-1}}{V'_1}^*)=1_{C(S^{2m-1})\otimes \mathcal{O}_{n+1}}$,
and the function $\lambda\colon [0,1]\to S^1$ is continuous.
Now we get the element $V\colon =(\bar{\lambda}c+1_{C([0,1]\times S^{2m-1})}\otimes e)V'\in U_{(C([0,1]\times S^{2m-1})\otimes E_{n+1})}$ satisfying the assertion of the lemma.
Since $V_1(1_{C(S^{2m-1})}\otimes \mathbb{T})=u_{{\alpha'}^{-1}}(1_{C(S^{2m-1})}\otimes \mathbb{T})=\alpha'^{-1}(1_{C(S^{2m-1})}\otimes \mathbb{T})$ and $\alpha'(1_{C(S^{2m-1})}\otimes e)=(1_{C(S^{2m-1})}\otimes e)$,
direct computation yields
\begin{align*}
\alpha'\otimes {\rm id}_{\mathbb{M}_{2n+4}}\left(
\begin{array}{cccc}
0_1&V\mathbb{T}&e&0\\
0&0_{n+1}&0&0_{n+1}\\
&&0_1&0\\
\text{{\huge{O}}}_{n+2}&&\mathbb{T}^*V^*&0_{n+1}
\end{array}
\right)=\left(
\begin{array}{cccc}
0_1&\mathbb{T}&e&0\\
0&0_{n+1}&0&0_{n+1}\\
&&0_1&0\\
\text{{\huge{O}}}_{n+2}&&\mathbb{T}^*&0_{n+1}
\end{array}
\right).
\end{align*}
Hence $Y_{\alpha'}$ is an element of $\mathbb{M}_{2n+4}(B_{\alpha'})$ that is sent to $y_{\alpha'}$ by the quotient map $\pi_{\alpha'}\colon B_{\alpha'}\to A_{\alpha'}$. 
\end{proof}
We remark that $Y_{\alpha'}$ is a partial isometry.
We have $\theta_{\alpha'}(Y_{\alpha'})\theta_{\alpha'}(Y_{\alpha'})^*=1\oplus 0_{n+1}\oplus 0_1\oplus 1_{n+1}$ and $\theta_{\alpha'}(Y_{\alpha'})^*\theta_{\alpha'}(Y_{\alpha'})=0_1\oplus 1_{n+1}\oplus1\oplus 0_{n+1}$.
Recall that $W$ is a partial isometry with $WW^*=0_1\oplus 1_{n+1}\oplus 1\oplus 0_{n+1}$ and $W^*W=1\oplus 0_{n+1}\oplus 0_1\oplus 1_{n+1}$.
Therefore an element $W+\theta_{\alpha'}(Y_{\alpha'})$ is a unitary in $U_{\mathbb{M}_{2n+4}(M)}$. 
\begin{thm}\label{muoo}
Let $m\geq 1$ be a natural number.
Let $\alpha'$ be an element of $\operatorname{Map}(S^{2m-1}, \operatorname{Aut}_eE_{n+1})$ that is homotopic to $l$ in $\operatorname{Map}(S^{2m-1}, \operatorname{End}_eE_{n+1})$.
Let $w_{\alpha'},\, Y_{\alpha'}$ and $V_{\sigma_{\alpha'}}$ be as mentioned above.
Then we have $[w_{\alpha'}]_1=0$ in $K_1(\mathcal{Q}(C(S^{2m})\otimes \mathbb{K}))$.
\end{thm}
\begin{proof}
Recall the following commutative diagram
$$\xymatrix{
B_{\alpha'}\ar[d]\ar[rr]^{\pi_{\alpha'}}&&A_{\alpha'}\ar[dd]^{\tau_{\theta_{\alpha'}}}\\
M_{\alpha'}\ar[d]^{\theta_{\alpha'}}&&\\
M\ar[rr]^{\pi}&&\mathcal{Q}(C(S^{2m})\otimes \mathbb{K}).
}$$
By Lemma \ref{simpl} and Lemma \ref{implift},
we have
\begin{align*}
V_{\sigma_{\alpha'}}\oplus V_{\sigma_{\alpha'}}^*:=&\pi\otimes {\rm id}_{\mathbb{M}_{2n+4}}(W)+\tau_{\theta_{alpha'}}\otimes {\rm id}_{\mathbb{M}_{2n+4}}(y_{\alpha'})\\
=&\pi\otimes{\rm id}_{\mathbb{M}_{2n+4}}(W+\theta_{\alpha'}\otimes {\rm id}_{\mathbb{M}_{2n+4}}(Y_{\alpha'})),
\end{align*}
so $W+\theta_{\alpha'}(Y_{\alpha'})$ is a unitary lift of $V_{\sigma_{\alpha'}}\oplus V_{\sigma_{\alpha'}}^*$.
Let $P$ be a projection of the form
$$P\colon =(W+\theta_{\alpha'}(Y_{\alpha'}))(1_{n+2}\oplus 0_{n+2})(W+\theta_{\alpha'}(Y_{\alpha'}))^*.$$
We have ${\rm ind}[V_{\sigma_{\alpha'}}V^*_{\sigma}]_1={\rm ind}[V_{\sigma_{\alpha'}}]_1-{\rm ind}[V_{\sigma}]_1=[P]_0+[1_{C(S^{2m})}\otimes e]_0-[1_{n+2}]_0\in K_0((C(S^{2m})\otimes \mathbb{K})^{\sim})$ and we show that the index is 0.
%An element
%\begin{align*}
%Z_{\alpha'}\colon =&\theta_{\alpha'}\otimes {\rm id}_{\mathbb{M}_{2n+4}}(Y_{\alpha'})(1_{n+2}\oplus 0_{n+2})\\
%=&\theta_{\alpha'}\otimes {\rm id}_{\mathbb{M}_{2n+4}}\left(\left(
%\begin{array}{cccc}
%0_1&\mathbb{T}&&\text{{\huge{O}}}_{n+2}\\
%0&0_{n+1}&&\\
%&&&\\
%\text{{\huge{O}}}_{n+2}&&&\text{{\huge{O}}}_{n+2}
%\end{array}
%\right), \left(
%\begin{array}{cccc}
%0_1&V\mathbb{T}&&\text{{\huge{O}}}_{n+2}\\
%0&0_{n+1}\\
%&&&\\
%\text{{\huge{O}}}_{n+2}&&&\text{{\huge{O}}}_{n+2}
%\end{array}
%\right)
%\right)
%\end{align*}
%is a partial isometry with $Z_{\alpha'}Z^*_{\alpha'}=(1-e)\oplus 0_{n+1}\oplus 0_1\oplus 0_{n+1}$ and $Z^*_{\alpha'}Z_{\alpha'}=0_1\oplus 1_{n+1}\oplus 0_1\oplus 0_{n+1}$.
%Hence we have $\theta_{\alpha'}(1_{n+2}\oplus 0_{n+2})W^*=Z_{\alpha'}W^*=0$ and $W(1_{n+2}\oplus 0_{n+2})\theta_{\alpha'}\otimes {\rm id}_{\mathbb{M}_{2n+4}}(Y_{\alpha'})=WZ_{\alpha'}^*=0$.
Recall $V_t(1_{C(S^{2m-1})}\otimes e)=(1_{C(S^{2m-1})}\otimes e)V_t=(1_{C(S^{2m-1})}\otimes e)$ and $U_t(1_{C(S^{2m-1})}\otimes e)=(1_{C(S^{2m-1})}\otimes e)U_t=(1_{C(S^{2m-1})}\otimes e)$ by Definition \ref{iru}.
Direct computation yields
\begin{align*}
P=&W(1_{n+2}\oplus 0_{n+2})W^*+\theta_{\alpha'}\otimes{\rm id}_{\mathbb{M}_{2n+4}}(Y_{\alpha'}(1_{n+2}\oplus 0_{n+2})Y_{\alpha'}^*)\\
=&0_1\oplus 1_{n+1}\oplus 0_{n+2}+(1-e)\oplus0_{n+1}\oplus 0_1 \oplus 0_{n+1}
\end{align*}
Now we have
$P=(1-e)\oplus 1_{n+1}\oplus 0_{n+2}$
and get ${\rm ind}[V_{\sigma_{\alpha'}}V^*_{\sigma}]_1=[P]_0+[1_{C(S^{2m})}\otimes e]_0-[1_{n+2}]_0=0$.
Therefore we have $-n[w_{\alpha'}]_1=[V_{\sigma_{\alpha'}}V^*_{\sigma}]_1=0$ and this proves the theorem because $\operatorname{Tor}(K_1(\mathcal{Q}(C(S^{2m})\otimes \mathbb{K})), \mathbb{Z}_n)=0$.
\end{proof}
\begin{cor}\label{inj}
For the above $\alpha'$, we have $[\alpha']=0$ in $[S^{2m-1}, \operatorname{Aut}E_{n+1}]$.
\end{cor}
\begin{proof}
Since $[w_{\alpha'}]_1=0$, there is a unitary $W_{\alpha'}$ in $U_{\mathcal{M}(C(S^{2m})\otimes \mathbb{K})}$ that is a lift of $w_{\alpha'}$. 
Therefore we have $C(S^{2m})$-linear isomorphism ${\rm Ad}W_{\alpha'} \colon B_l\to B_{\alpha'}$, and $\mathcal{P}_l\cong \mathcal{P}_{\alpha'}$ from  Lemma \ref{bun}. 
By Lemma \ref{BG}, we have $[\alpha']=[l]=0$.
\end{proof} 

\begin{lem}\label{imp}
Let $m\geq 1$ be a natural number. Let $\alpha$ be an element in $\operatorname{Map}(S^{2m-1}, \operatorname{Aut}_eE_{n+1})$. If $\alpha \sim_h l$ in $\operatorname{Map}(S^{2m-1}, \operatorname{End}_0E_{n+1})$, then there exists $\alpha'$ in $\operatorname{Map}(S^{2m-1}, \operatorname{Aut}_eE_{n+1})$ satisfying the following :
$$\alpha'\sim_h l\; in\; \operatorname{Map}(S^{2m-1}, \operatorname{End}_eE_{n+1}),\;
\alpha\sim_h \alpha'\; in \; \operatorname{Map}(S^{2m-1}, \operatorname{Aut}E_{n+1}).$$
\end{lem}
\begin{proof}
It follows from Lemma \ref{LE} that there is an exact sequence 
$$[S^{2m}, \operatorname{B}S^1] \to [S^{2m-1}, \operatorname{End}_eE_{n+1}]\to [S^{2m-1}, \operatorname{End}_0E_{n+1}]\to [S^{2m-1}, \operatorname{B}S^1],$$
and by Remark \ref{hs} the map $[S^{2m-1}, \operatorname{End}_eE_{n+1}]\to [S^{2m-1}, U_{E_{n+1}}]$ which sends         $[\alpha]$ to $[u_{\alpha}] :=[e+\sum_{i=1}^{n+1}\alpha_x(T_i)T_i^*] \in[S^1, U_{E_{n+1}}]$ is an isomorphism. Since $\operatorname{B}S^1$ is $K(\mathbb{Z}, 2)$ space, if $m\geq 2$ and $\alpha \sim_h l$ in $\operatorname{Map}(S^{2m-1}, \operatorname{End}_0E_{n+1})$, then we have $[\alpha]=0$ in $[S^{2m-1}, \operatorname{End}_eE_{n+1}]$ because $[S^{2m-1}, \operatorname{End}_eE_{n+1}]=[S^{2m-1}, \operatorname{End}_0E_{n+1}]$, $m\geq 2$.\\
Hence it is sufficient to show the claim in the case of $m=1$.
Let $\alpha$ be an element of $\operatorname{Map}(S^1, \operatorname{Aut}_eE_{n+1})$ with $\alpha\sim_h l$ in $\operatorname{Map}(S^1, \operatorname{End}_0E_{n+1})$. 
Computation in Theorem \ref{homotopy} yields that there exists $d\in \mathbb{Z}$ with $[u_{\alpha}]=-nd \in[S^1, U_{E_{n+1}}]=\mathbb{Z}$. We define $\rho_d$ by 
$$\rho_d \colon={\rm Ad}(\bar{z}^dT_1T_1^*+(1-T_1T_1^*)) \in \operatorname{Map}(S^1, \operatorname{Aut}_eE_{n+1}).$$ 
By Lemma \ref{k1}, there is a continuous path from $(\bar{z}^dT_1T_1^*+(1-T_1T_1^*))$ to $\bar{z}^d$ in $U_{C(S^1)\otimes E_{n+1}}$, and $\rho_d\sim_h {\rm Ad}\bar{z}^d=l$ in $\operatorname{Map}(S^1, \operatorname{Aut}E_{n+1})$. 
We have $[u_{\rho_d\alpha}]_1=[\rho_d(u_{\alpha})]_1+[u_{\rho_d}]_1$ in $K_1(C(S^1)\otimes E_{n+1})=[S^1, U_{E_{n+1}}]$. 
By Lemma \ref{tri}, it follows that $[\rho_d(u_{\alpha})]_1=[u_{\alpha}]_1$. Hence we have $[u_{\rho_d\alpha}]_1=-nd+[u_{\rho_d}]_1$. The following computations yield $[u_{\rho_d}]_1=nd$ :
\begin{align*}
u_{\rho_d}&=e+\sum_{i=1}^{n+1}(\bar{z}^dT_1T_1^*+(1-T_1T_1^*))T_i(z^dT_1T_1^*+(1-T_1T_1^*))T_i^*)\\
&=(\bar{z}^dT_1T_1^*+(1-T_1T_1^*))(e+\sum_{i=1}^{n+1}T_i(z^dT_1T_1^*+(1-T_1T_1^*))T_i^*),
\end{align*}
\begin{align*}
&\left(
\begin{array}{cc}
e+\sum_{i=1}^{n+1}T_i(z^dT_1T_1^*+(1-T_1T_1^*))T_i^*)&0\\
0&1_{n+1}
\end{array}\right)\\
=&\left(
\begin{array}{cccc}
\mathbb{T}& e\\
0_{n+1}&\mathbb{T}^*
\end{array}\right)\left(
\begin{array}{ccc}
(z^dT_1T_1^*+(1-T_1T_1^*))\otimes 1_{n+1}&0 \\
0& 1
\end{array}\right)\left(
\begin{array}{cc}
\mathbb{T}^*&0_{n+1}\\
e&\mathbb{T}
\end{array}\right).
\end{align*}
Therefore we have $[u_{\rho_d\alpha}]=0$ in $[S^1, U_{E_{n+1}}]$. 
By Remark \ref{hs}, we have the isomorphism $[S^1, \operatorname{End}_eE_{n+1}]\ni \rho_d\alpha \mapsto [u_{\rho_d\alpha}] \in [S^1, U_{E_{n+1}}]$, and $\alpha'\colon=\rho_d\alpha$ satisfies all assumptions of the lemma.
\end{proof}
We show the weak homotopy equivalence.
\begin{thm}\label{main}
The inclusion map $\operatorname{Aut}E_{n+1} \to \operatorname{End}_0E_{n+1}$ is a weak homotopy equivalence.
\end{thm}
\begin{proof}
By  Lemma \ref{even} and Theorem \ref{homotopy}, we consider only the case of odd homotopy groups. 
Let $k$ be an odd number.

First, we show the map $[S^k, \operatorname{Aut}E_{n+1}]\to[S^k, \operatorname{End}_0E_{n+1}]$ is injective. 
If $\alpha$ in $\operatorname{Map}(S^k, \operatorname{Aut}E_{n+1})$ is homotopic to $l$ in $\operatorname{Map}(S^k, \operatorname{End}_0E_{n+1})$, we may assume that there exists $\alpha'\in \operatorname{Map}(S^k, \operatorname{Aut}_eE_{n+1})$ homotopic to $\alpha$ by Remark \ref{min}.
From  Lemma \ref{imp}, we may assume that $\alpha'\sim_h l$ in $\operatorname{Map}(S^k, \operatorname{End}_eE_{n+1})$, and we have $[\alpha']=[\alpha]=0$ in $[S^k, \operatorname{Aut}E_{n+1}]$ by Corollary \ref{inj}. 
Therefore the map $[S^k, \operatorname{Aut}E_{n+1}]\to [S^k, \operatorname{End}_0E_{n+1}]$ is injective.

Second, we show the surjectivity.
The following commutative diagram holds
$$\xymatrix{
[S^k, \operatorname{Aut}E_{n+1}]\ar@{->>}[r]^{{\rm Lem}\; \ref{sur}}\ar[d]&[S^k, \operatorname{Aut}\mathcal{O}_{n+1}]\ar@{=}[d]\\
[S^k, \operatorname{End}_0E_{n+1}]\ar[r]&[S^k, \operatorname{End}\mathcal{O}_{n+1}].
}$$
In the case of $k=1$, we have $[S^1, \operatorname{End}_0E_{n+1}] = [S^1, \operatorname{End}\mathcal{O}_{n+1}]=\mathbb{Z}_n$ because the generators of the both groups are constructed from canonical gauge actions of $S^1$ that are of the form $\lambda_z \colon T_i\mapsto zT_i$ and $\tilde{\lambda}_z \colon S_i\mapsto zS_i$. Therefore the surjectivity follows from  Lemma \ref{sur}.

In the case of $k\geq 3$, the map $[S^k, U_{E_{n+1}}]\to [S^k, \operatorname{End}_0E_{n+1}]=\mathbb{Z}$ in Theorem \ref{homotopy} is an isomorphism. 
Therefore the map $\mathbb{Z}=[S^k, \operatorname{End}_0E_{n+1}]\to [S^k, \operatorname{End}\mathcal{O}_{n+1}]=[S^k, U_{\mathcal{O}_{n+1}}]=\mathbb{Z}_n$ is the quotient by $n\mathbb{Z}$. 
Hence the image of the map $[S^k, \operatorname{Aut}E_{n+1}]\to [S^k, \operatorname{End}_0E_{n+1}]=\mathbb{Z}$ contains an element  $nd+1$ for some $d\in \mathbb{Z}$ by Lemma \ref{sur}.

On the other hand, we show that the image contains $n\mathbb{Z}$. 
For every $V\in U_{C(S^k)\otimes E_{n+1}}$, there exists $V'\in U_{C(S^k)\otimes E_{n+1}}$ with $V'(1_{C(S^k)}\otimes e)=(1_{C(S^k)}\otimes e)V'=(1_{C(S^k)}\otimes e)$ which is homotopic to $V$ in $U_{C(S^k)\otimes E_{n+1}}$ by Remark \ref{hs}.

Since the isomorphism $[S^k, U_{E_{n+1}}]\to [S^k, \operatorname{End}_0^1E_{n+1}]$ sends $$-n[V]_1=[1_{C(S^k)}\otimes e+\sum_{i=1}^{n+1}V'(1\otimes T_i)V'^*(1\otimes T_i^*)]_1$$ to $[{\rm Ad}V']=[{\rm Ad}V]$,   
the subset $$\{ [{\rm Ad}V] \in [S^k, \operatorname{Aut}E_{n+1}] \mid V\in U_{C(S^k)\otimes E_{n+1}}\}$$ is mapped onto the subset $$\{-n[V]_1 \in K_1(C(S^k)\otimes E_{n+1}) \mid V\in U_{(C(S^k)\otimes E_{n+1}}\}=n\mathbb{Z} \subset [S^k, \operatorname{End}_0E_{n+1}]=\mathbb{Z}.$$  Therefore the image contains $nd+1$ and $n\mathbb{Z}$, and we have the conclusion.
\end{proof}

\subsection{An exact sequence of homotopy sets}
 We have the principal $\operatorname{Aut}_eE_{n+1}$-bundle
$\operatorname{Aut}_eE_{n+1}\xrightarrow{i}\operatorname{Aut}E_{n+1}\xrightarrow{\eta}\operatorname{B}S^1.$ We denote by $f$ the classifying map of the bundle and denote by $r$ the restriction map $\operatorname{Aut}E_{n+1}\to\operatorname{Aut}\mathbb{K}$. In this section, we show the following theorem.
\begin{thm}\label{seq}
Let $X$ be a compact CW-complex. Then we have the following exact sequence of the pointed set where first $4$-terms gives the exact sequence of the groups :
$$H^1(X)\to K^1(X)\to [X, \operatorname{Aut}E_{n+1}]\xrightarrow{\eta_*} H^2(X)\xrightarrow{f_*}[X, \operatorname{BAut}_eE_{n+1}]\xrightarrow{Bi_*}[X, \operatorname{BAut}E_{n+1}]\xrightarrow{Br_*}H^3(X).$$
It follows that ${\rm Im}\; \eta_*\subset \operatorname{Tor}(H^2(X), \mathbb{Z}_n)$ and ${\rm Im}\; Br_*\subset \operatorname{Tor}(H^3(X), \mathbb{Z}_n)$.
\end{thm}
The following lemma is well known in the homotopy theory.
We refer to \cite[Chap 3, Section 6]{TL}
\begin{lem}\label{gen}
Let $X$ be a CW-complex.
Let $G$ be a topological group and let $H$ be a subgroup of $G$ such that $H\to G\to G/H$ is a principal $H$-bundle.
Suppose that $G/H$ has a homotopy type of a CW-complex.
Let $f : G/H \to\operatorname{B}H$ be its classifying map. Then we have the exact sequence of  pointed sets :
$$[X, G]\to[X, G/H]\xrightarrow{f_*}[X, \operatorname{B}H]\to[X, \operatorname{B}G].$$
\end{lem}
Since $\operatorname{B}S^1$ has a homotopy type of a CW-complex, we can apply the above lemma to $\operatorname{Aut}_eE_{n+1}\to \operatorname{Aut}E_{n+1}\to \operatorname{B}S^1.$

\begin{lem}\label{thr}
Let $X$ be a CW-complex.
The following sequence of pointed sets is exact :
$$[X, \operatorname{BAut}_eE_{n+1}]\xrightarrow{Bi_*}[X, \operatorname{BAut}E_{n+1}]\xrightarrow{Br_*}[X, \operatorname{BAut}\mathbb{K}].$$
\end{lem}
\begin{proof}
The group $\operatorname{Aut}_e\mathbb{K}$ is identified with the group $U_{\mathcal{M}(\mathbb{K})}$ by the map taking the implementing unitary $U_{\alpha}=\sum_{i\neq 1}\alpha(e_{i1})e_{1i}$ for $\alpha \in \operatorname{Aut}_{e_{11}}\mathbb{K}$. Hence it is contractible, and $[X, \operatorname{BAut}_e\mathbb{K}]=\{pt\}$.
From the commutative diagram below, 
$$\xymatrix{
[X, \operatorname{BAut}_eE_{n+1}]\ar[r]\ar[rd]& [X, \operatorname{BAut}E_{n+1}]\ar[r]&[X, \operatorname{BAut}\mathbb{K}]\\
&[X, \operatorname{BAut}_e\mathbb{K}],\ar[ur]&
}$$
$Br_*\circ Bi_*$ is trivial.
Therefore it is sufficient to prove that for every $\mathcal{P}\in \operatorname{Map}(X, \operatorname{BAut}E_{n+1})$ with the trivial associated bundle $\mathcal{P}\times_{\operatorname{Aut}E_{n+1}}\operatorname{Aut}\mathbb{K}$, the structure group of $\mathcal{P}$ is reduced to $\operatorname{Aut}_eE_{n+1}$. 
Let $\mathcal{P}\in \operatorname{Map}(X, \operatorname{BAut}E_{n+1})$ be a principal $\operatorname{Aut}E_{n+1}$-bundle with the trivial associated bundle $\mathcal{P}\times_{\operatorname{Aut}E_{n+1}}\operatorname{Aut}\mathbb{K}$. 
We take an open covering $\{U_i\}$ of $X$ giving a local trivialization of $\mathcal{P}$, and denote by $\phi_{ji} : U_j\cup U_i\to \operatorname{Aut}E_{n+1}$ the transition function. 
By the assumption, there exists the map $h_i :U_i\times \operatorname{Aut}\mathbb{K}\to U_i\times \operatorname{Aut}\mathbb{K}$ that is compatible with the transition functions, and is equivariant with respect to the right multiplication of $\operatorname{Aut}\mathbb{K}$. 
The diagram below holds 
$$\xymatrix{
U_i\cap U_j\times \operatorname{Aut}\mathbb{K}\ar[d]^{r(\phi_{ji})}\ar[r]^{h_i}&U_i\cap U_j\times \operatorname{Aut}\mathbb{K}\ar@{=}[d]\\
U_j\cap U_i\times \operatorname{Aut}\mathbb{K}\ar[r]^{h_j}&U_j\cap U_i\times \operatorname{Aut}\mathbb{K}.
}$$We also denote by $\phi_{ji}$ the map 
$$U_i\cap U_j\times \operatorname{Aut}\mathbb{K}\ni (x, \alpha)\mapsto (x,r(\phi_{ji}(x))\alpha)\in U_j\cap U_i\times \operatorname{Aut}\mathbb{K}.$$
We denote $h_i(x):=Pr_i(h_i(x, {\rm id}))$ where $Pr_i : U_i\times \operatorname{Aut}\mathbb{K}\to \operatorname{Aut}\mathbb{K}$. Since $h_i$ is equivariant, we have $h_i^{-1}(x)=h_i(x)^{-1}$.
We have $h_j(x)r(\phi_{ji}(x)) h_i^{-1}(x)(e)=e$ because $h_j\circ \phi_{ji}\circ h_i^{-1}(x, {\rm id})=(x, {\rm id})$ for every $x\in U_j\cap U_i$.
 If we take an appropriate refinement of $\{U_i\}$, we  may assume that for every $i$, there exists $x_i \in U_i$ satisfying $||h^{-1}_i(x)(e)-h^{-1}_i(x_i)(e)||<1, \; x\in U_i$.  There is a unitary $V'_i(x)$ that is the sum of partial isometries constructed from the polar decomposition of $h_i^{-1}(x)(e)h_i^{-1}(x_i)(e)$ and $(1-h_i^{-1}(x)(e))(1-h^{-1}(x_i)(e))$, and $V'_i(x)h^{-1}_i(x_i)(e){V_i'(x)}^*=h_i^{-1}(x)(e)$ holds. We fix a unitary $W_i\in U_{\mathbb{K}^1}$ with $W_ieW_i^*=h_i^{-1}(x_i)(e)$. Then we have a unitary $V_i(x)=V'_i(x)W_i\in U_{\mathbb{K}^{\sim}}$ with $V_i(x)eV_i(x)^*=h_i^{-1}(x)(e)$. The correction of the map 
$$u_i : U_i\times \operatorname{Aut}E_{n+1} \ni (x, \alpha)\mapsto (x, {\rm Ad}V_i\alpha)\in U_i\times \operatorname{Aut}E_{n+1}$$
gives the following :
$$\xymatrix{
U_i\cap U_j\times \operatorname{Aut}E_{n+1}\ar[d]^{\phi_{ji}}&U_i\cap U_j\times \operatorname{Aut}E_{n+1}\ar[l]^{u_i}\ar[d]^{\tilde{\phi}_{ji}}\\
U_j\cap U_i\times \operatorname{Aut}E_{n+1}&U_j\cap U_i\times \operatorname{Aut}E_{n+1},\ar[l]^{u_j}
}$$
where $\tilde{\phi}_{ji}$ is of the form
$$\tilde{\phi}_{ji} : (x,\alpha)\mapsto (x, {\rm Ad}V_j(x)^*\phi_{ji}(x){\rm Ad}V_i(x)\alpha).$$  We have the transition function $$U_j\cap U_i\ni x\mapsto {\rm Ad}V_j(x)^*\phi_{ji}(x){\rm Ad}V_i(x)\in\operatorname{Aut}_eE_{n+1}$$
by the computation below :
\begin{align*}
 {\rm Ad}V_j(x)^*\phi_ji(x){\rm Ad}V_i(x)(e)=&{\rm Ad}V_j(x)^*\phi_{ji}(x)h_j^{-1}(x)(e)\\
=&{\rm Ad}V_j(x)^*h_j^{-1}(x)(e)\\
=&{\rm Ad}V_j(x)^*{\rm Ad}V_j(x)(e)\\
=&e.
\end{align*}
Therefore the structure group of $\mathcal{P}$ is reduced to $\operatorname{Aut}_eE_{n+1}$.
\end{proof}

\begin{lem}\label{ho}
Let $X$ be a compact Hausdorff space.
The map $\operatorname{Aut}_eE_{n+1}\ni \alpha \mapsto e+\sum_{i=1}^{n+1}\alpha(T_i)T_i^* \in U_{E_{n+1}}$ induces a group isomorphism $[X, \operatorname{Aut}_eE_{n+1}]\to [X, U_{E_{n+1}}]=K^1(X)$.
\end{lem}
\begin{proof}
By Remark \ref{Ae} and Theorem \ref{main}, the map is bijective. 
So we show that it is a group homomorphism.
Let $\alpha$ and $\beta$ be elements of $\operatorname{Map}(X, \operatorname{Aut}_eE_{n+1})$, and we denote $u_\alpha \colon=1_{C(X)}\otimes e+\sum_{i=1}^{n+1}\alpha(1_{C(X)}\otimes T_i)1_{C(X)}\otimes T_i^*\in U_{C(X)\otimes E_{n+1}}$. We show $[u_{\alpha\beta}]_1=[u_{\alpha}]_1+[u_{\beta}]_1$.
Since $\alpha$ and $\beta$ fix $e$, direct computation yields
\begin{align*}
u_{\alpha\beta}=&\alpha(e+\sum_i\beta(1_{C(X)}\otimes T_i)(1_{C(X)}\otimes T_i^*))(e+\sum_i\alpha(1_{C(X)}\otimes T_i)(1_{C(X)}\otimes T_i^*))\\
=&\alpha(u_{\beta})u_{\alpha}.
\end{align*}
By Lemma \ref{tri}, we have $[\alpha(u_{\beta})]_1=K_1(\alpha)([u_{\beta}]_1)=[u_{\beta}]_1$.
\end{proof}
We need the following fact to determine the second cohomology group of $\operatorname{Aut}E_{n+1}$. See Allen Hatcher's unpublished book \cite[Proposition 5.11]{Hat}.
\begin{prop}
Let $X$ be a path connected space with finite homotopy groups. Then its homology group $H_n(X)$ is finite for all $n>0$.
\end{prop}\label{hat}
\begin{lem}\label{h}
We have the following cohomology groups : 
$$H^2(\operatorname{Aut}E_{n+1})=\mathbb{Z}_n,\; H^3(\operatorname{BAut}E_{n+1})=\mathbb{Z}_n.$$
\end{lem}
\begin{proof}
Two spaces $\operatorname{Aut}E_{n+1}$ and $\operatorname{Aut}\mathcal{O}_{n+1}$ are path connected and the map $\operatorname{Aut}E_{n+1} \to\operatorname{Aut}\mathcal{O}_{n+1}$ gives 
\begin{align*}
\pi_i(\operatorname{Aut}E_{n+1})\cong&\pi_i(\operatorname{Aut}\mathcal{O}_{n+1})\;i=0,\;1,\;2\\
\pi_3(\operatorname{Aut}E_{n+1})\twoheadrightarrow& \pi_3(\operatorname{Aut}\mathcal{O}_{n+1}).
\end{align*}
So we have 
\begin{align*}
H_i(\operatorname{Aut}E_{n+1})\cong&H_i(\operatorname{Aut}\mathcal{O}_{n+1})\; i=0,\;1,\;2\\
H_3(\operatorname{Aut}E_{n+1})\twoheadrightarrow& H_3(\operatorname{Aut}\mathcal{O}_{n+1}).
\end{align*}
by Whitehead's theorem (see \cite[Corollary 6.69]{AT}). By the universal coefficient theorem, we have
$$H^2(\operatorname{Aut}E_{n+1})\cong free(H_2(\operatorname{Aut}E_{n+1}))\oplus \operatorname{Tor}(H_1(\operatorname{Aut}E_{n+1}))$$
where $free(H_2(\operatorname{Aut}E_{n+1}))$ is the free part of the homology group. By Proposition \ref{hat}, the homology groups of $\operatorname{Aut}\mathcal{O}_{n+1}$ are finite, and  $free(H_2(\operatorname{Aut}E_{n+1}))=0$. So we have $H^2(\operatorname{Aut}E_{n+1})\cong\operatorname{Tor}(H_1(\operatorname{Aut}\mathcal{O}_{n+1}))$.  Hurewicz' theorem ( \cite[Theorem 6.66]{AT}) yields $H_1(\operatorname{Aut}\mathcal{O}_{n+1})=\pi_1(\operatorname{Aut}\mathcal{O}_{n+1})=\mathbb{Z}_n$.
Similarly, we have $H^3(\operatorname{BAut}E_{n+1})=\mathbb{Z}_n$. 
\end{proof}

Now we prove Theorem \ref{seq}.
\begin{proof}[Proof of Theorem \ref{seq}]
By Lemma \ref{gen} and the long exact sequences of the principal bundle $\operatorname{Aut}_eE_{n+1}\xrightarrow{i}\operatorname{Aut}E_{n+1}\xrightarrow{\eta}\operatorname{B}S^1$, we have an exact sequence of pointed set where first $4$-terms gives the exact sequence of the groups : 
$$H^1(X)=[X, S^1]\to [X, \operatorname{Aut}_eE_{n+1}]\to [X, \operatorname{Aut}E_{n+1}]\xrightarrow{\eta_*} H^2(X)\xrightarrow{f_*}[X, \operatorname{BAut}_eE_{n+1}]\xrightarrow{Bi_*}[X, \operatorname{BAut}E_{n+1}].$$ 
By Lemma \ref{ho}, and Lemma \ref{thr}, we have the exact sequence :
$$H^1(X)\to K^1(X)\to [X, \operatorname{Aut}E_{n+1}]\xrightarrow{\eta_*} H^2(X)\xrightarrow{f_*}[X, \operatorname{BAut}_eE_{n+1}]\xrightarrow{Bi_*}[X, \operatorname{BAut}E_{n+1}]\xrightarrow{Br_*}H^3(X),$$
where we identify $H^3(X)$ with $[X, \operatorname{BAut}\mathbb{K}]$ because $\operatorname{BAut}\mathbb{K}$ is the $K(\mathbb{Z}, 3)$-space.
We identify $[\operatorname{Aut}E_{n+1}, \operatorname{B}S^1]$ with $H^2(\operatorname{Aut}E_{n+1})$.
For every $[\alpha] \in [X, \operatorname{Aut}E_{n+1}]$, it follows that $\eta_*([\alpha])=\alpha^*([\eta])$, and the element $\alpha^*([\eta])$ is in the image of the map $$\alpha^* \colon H^2(\operatorname{Aut}E_{n+1})=[\operatorname{Aut}E_{n+1}, \operatorname{B}S^1] \ni [\eta] \mapsto  [\eta\circ\alpha] \in [X, \operatorname{B}S^1]=H^2(X).$$ Therefore we have $\operatorname{Im}\eta_*\subset \operatorname{Tor}(H^2(X), \mathbb{Z}_n)$ from Lemma \ref{h}. 
%Since $\operatorname{BAut}\mathbb{K}$ is the $K(\mathbb{Z}, 3)$-space, 
Similar argument yields $\operatorname{Im}Br_*\subset \operatorname{Tor}(H^3(X), \mathbb{Z}_n)$.
\end{proof}
\section*{Acknowledgements}
The author would like to show his greatest appreciation to his supervisor Prof. Masaki Izumi who gave many insightful comments and suggestions, and patiently checked his arguments.

\end{document}